\documentclass{article}

\usepackage[latin1]{inputenc}
\usepackage{fullpage}
\usepackage{latexsym,amssymb,amsmath,amsfonts}
\usepackage{graphicx}
\usepackage{pst-all}

\newtheorem{theorem}{Theorem}
\newtheorem{proposition}[theorem]{Proposition}
\newtheorem{lemma}[theorem]{Lemma}
\newtheorem{example}[theorem]{Example}

  \newenvironment{proof}{\small{\bf Proof.}}%
  {\hfill$\Box$\normalsize\bigskip}

\allowdisplaybreaks

\newcommand{\RR}{{\mathbb R}} 
 
\newcommand{\EE}{{\mathbb E\,}} 
\newcommand{\PP}{{\mathbb P}} 
\newcommand{\1}{{\mathbf 1}}
\newcommand{\NATURE}{K}
\newcommand{\nature}{k} 
\newcommand{\ACTIONS}{{A}}
\newcommand{\Action}{A} 
\newcommand{\action}{a} 
\newcommand{\SIGNED}{\Sigma} 

\newcommand{\signed}{s}
\newcommand{\BELIEFS}{\Delta}
\newcommand{\belief}{p}
\newcommand{\beliefbis}{q}
\newcommand{\priorbelief}{\bar{p}} 
\newcommand{\sphere}{S^{|\NATURE|-1}}
\newcommand{\seminorm}[1]{\Vert #1 \Vert} 
\newcommand{\FACE}{F}
\newcommand{\face}{f}
\newcommand{\NORMAL}{N}
\newcommand{\normal}{n}
\newcommand{\neighborhood}[1]{{\mathcal #1}}
\newcommand{\OptimalActions}{\ACTIONS^\star}
\newcommand{\RevealedBeliefs}{\BELIEFS_{\ACTIONS}^\star}
\newcommand{\ConfidenceSet}{\BELIEFS_{\ACTIONS}^{\textrm{c}}}

\newcommand{\IndifferenceKernel}{\SIGNED_{\ACTIONS}^{\textrm{i}}}
\newcommand{\ConstantSupspeciale}{C_{\ACTIONS}}
\newcommand{\ConstantSup}{C_{\priorbelief,\ACTIONS}}
\newcommand{\ConstantInf}{c_{\priorbelief,\ACTIONS}}
\newcommand{\ConstantInfEpsilon}{c_{\priorbelief,\ACTIONS,\varepsilon}}
\newcommand{\mtext}[1]{\,\mbox{#1}\,} 

\newcommand{\proscal}[2]{\left\langle#1\:,#2\right\rangle}  
\newcommand{\norm}[1]{\|#1\|}

\newcommand{\opt}{^{\sharp}}

\newcommand{\Value}{{v}}

\newcommand{\wealth}{w}
\newcommand{\riskaversion}{R}
\newcommand{\utility}{u}
\newcommand{\Utility}{U}
\newcommand{\defset}[2]{\left\{#1\:\left|\:#2\right.\right\}}
\newcommand{\np}[1]{(#1)}                                   
\newcommand{\bp}[1]{\big(#1\big)}                           
\newcommand{\Bp}[1]{\Big(#1\Big)}                           
\newcommand{\bc}[1]{\big[#1\big]}                           
\newcommand{\Bc}[1]{\Big[#1\Big]}                           
\newcommand{\Ba}[1]{\Big\{#1\Big\}}                         
\newcommand{\eqsepv}{\; , \enspace}       
\newcommand{\eqfinv}{\; ,}                
\newcommand{\eqfinp}{\; .}                

\newcommand{\va}[1]{\mathbf{#1}}
\newcommand\VoI{\mathbf{VoI}}
\newcommand{\Primal}{\RR^{\NATURE}}
\newcommand{\Dual}{\RR^{\NATURE}} 
\newcommand{\primal}{x}
\newcommand{\dual}{y}
\newcommand{\SUBSET}{S}

\newcommand{\Convex}{C}
\newcommand{\aff}{\mathop{\mathrm{aff}}}

\newcommand{\ri}{\mathop{\mathrm{ri}}}

\title{Payoffs-Beliefs Duality and the Value of Information}

\author{Michel \textsc{De Lara}\thanks{%
CERMICS, \'Ecole des Ponts, UPE, Champs-sur-Marne, France.
E-mail: \texttt{michel.delara@enpc.fr}}
\and Olivier \textsc{Gossner}\thanks{%
CREST, CNRS, \'Ecole Polytechnique,
and London School of Economics. Email: \texttt{olivier.gossner@cnrs.fr}}}

\begin{document}

\maketitle

    \begin{abstract}%
    In decision problems under incomplete information, 
actions (identified to payoff vectors indexed by states of nature) 
and beliefs are naturally paired by bilinear duality.
    We exploit this duality to analyze the value of information,
    using concepts and tools from convex analysis.
    We define the value function as the support function of the set of available actions:
    the subdifferential at a belief is the set of optimal actions at this belief;
    the set of beliefs at which an action is optimal is the normal cone 
of the set of available actions at this point.
    Our main results are 1) a necessary and sufficient condition for positive value of information
    2) global estimates of the value of information of any information
    structure from local properties of the value function 
    and of the set of optimal actions taken at the prior belief only.
    We apply our results to the marginal value of information at the null,
    that is, when the agent is close to receiving no information at all, 
and we provide conditions under which the marginal value of information is
infinite, null, or positive and finite. 
    \end{abstract}%




\section{Introduction}

 The value of a piece of information to an economic agent depends on the information at hand, on the agent's prior on the state of nature, and on the decision problem faced.   
These elements are intrinsically tied, and separating the influence of one of
them from that of the others is not straightforward.

Most information rankings are either uniform among agents or restricted to
certain classes of agents. Blackwell's comparison of experiments
\cite{BlackwellEquivalentComparisonOfExperiments53}, for instance, is uniform;
it states that an information structure is more informative than another if all
agents, no matter  their available choices and preferences, weakly prefer the
former to the latter. Papers \cite{Lehmann, PersicoAuctions2000,
  CabralesGossnerSerranoEntropyAER2013} are examples
that build information rankings based on restricted sets of decision problems. 
The flip side of this approach is that information rankings are silent as to the
dependency of the value of a fixed piece information on the agent's preferences
and available choices. They do not tell us what makes information more or less
valuable to an arbitrary agent, and neither can they identify the agents who
value a given piece of information more than others. If we want to answer this
type of questions, we need to examine carefully how information, priors,
decisions and preferences come into play.  

The effect of priors and evidence on beliefs is well understood. Given a prior
belief, and after receiving some information, an agent forms a posterior
belief. Posterior beliefs average out to the prior belief, and information
acquisition can usefully be represented by the distribution of these posterior
beliefs (see, e.g.\,\cite{BohnenblustShapleySherman49,AumMas67inproc}).

In any decision problem, to each decision and state of nature corresponds a
payoff. The decision problem can thus be represented as a set of available
vector payoffs, where each payoff is indexed by a state of nature
\cite{BlackwellComparisonOfExperiments51}. Given a posterior belief, the agent
makes a decision that maximizes her expected utility so that, to each 
(posterior) belief of the agent corresponds an expected utility at this
belief. The corresponding map from beliefs to expected payoffs is called the
\emph{value function}.
The value of a piece of information, defined as the difference in expected utilities from
having or not having the information at hand, is thus the difference between the
expectation of the value function at the posterior and at the prior, and is
nonnegative. 
Thus, the value function fully captures the agent's preferences for information. 

In this paper, we make use of \emph{convex analysis} \cite{Rockafellar:1970}
to exploit a bilinear duality structure between payoffs and beliefs,
that gives expected payoff \cite{Dentcheva-Ruszczynski:2013}. 
Primal variables are payoffs vectors, dual variables are beliefs (or, more
generally, signed measures) and the value function appears as the 
(restriction to beliefs of the) support function of the set of available vector
payoffs. This provides a correspondence between convex analysis concepts and
tools, on the one hand, and economic objects, on the other hand. 
The set of beliefs compatible with an optimal action
is related to the \emph{normal cone} of the set of available vector
payoffs at this optimal action.
The \emph{subdifferential} of the value function at any belief 
can be represented
as the set of optimal choice of vector payoffs at this belief. 

We express the value of information according to the influence it has on
decisions. 
We provide three upper and lower bounds on the value of information. 

In the first upper and lower bounds, we characterize information with a positive
value. We show that information has a positive value if and only if at least one
of the optimal actions at the prior becomes suboptimal for some of the
posteriors. We thus define the confidence set at a prior belief~$\priorbelief$
as the set of posterior beliefs for which all optimal actions at~$\priorbelief$
remain optimal. Our result says that information has positive value if and only
if posterior beliefs fall outside of the confidence set with positive
probability. This result generalizes insights from \cite{Hir71} and
\cite{mirman1993monopoly}, who had already noticed that information can only be
useful insofar as it influences choices. We provide corresponding lower and
upper bounds to the value of information.  
 
In the second bounds, we express the fact that the value of information is
maximal when it influences actions the most, which happens when information
breaks indifferences between several choices. We show that, when this is the
case, the value of information can be suitably measured by an expected distance
between the prior and the posterior. There are several optimal actions at the
prior, and information that allows to break indifferences has highest value.  

Finally, our third bounds apply to cases in which the agent's optimal choice is
a smooth function of her belief around the prior. We show that, in this
situation, the value function is also smooth around the prior, and the value of
information is essentially a quadratic function of the expected distance between
the prior and the posterior. In this intermediate case, information impacts
actions in a continuous way. The optimal actions at the prior belief and at a
posterior close to it are themselves close; so choosing one instead of the other
has a mild, albeit positive, impact on the expected payoff. 

In a finite decision problem --- 
such as shopping behavior \cite{mcfadden1973conditional} 
or residential location \cite{mcfadden1978modeling} --- 
at any given prior the agent either has an optimal action that is locally
constant, or is indifferent between several optimal choices. The first and
second upper and lower bounds are particularly useful in finite choice
problems. The third bounds are most useful in decision problems with a continuum
of choices, such as scoring rules \cite{BrierScoringRules1950} or investment
decisions \cite{Arrow71}.  
\medskip

The paper is organized as follows. Sect.~\ref{sec:model} presents the model
and introduces the duality between actions/payoffs and beliefs. The
main results are presented in Sect.~\ref{sec:value_info}. 
Sect.~\ref{sec:insure} is devoted to an illustration of our results in an insurance example and
Sect.~\ref{sec:marginal} to 
applications to the question of marginal value of information.
Sect.~\ref{sec:related_literature}
concludes by discussing related literature.
The Appendix contains background on convex analysis and the proofs.

\section{Model, payoffs-beliefs duality and information}
\label{sec:model}

We consider the classical question of an agent who faces a decision problem 
under imperfect information on a state of nature. 
The  set of states of nature is a finite set~$\NATURE$. 
We identify the set~$\SIGNED$ of signed measures on~$\NATURE$ 
with~$\RR^{\NATURE}$.
The agent holds a prior belief~$\priorbelief$ with full support in the set 
\(   \BELIEFS=\Delta(\NATURE) \subset \SIGNED=\RR^{\NATURE} \) 
of probability distributions over~$\NATURE$. 
We identify~$\BELIEFS$ with the simplex of~$\RR^{\NATURE}$.

A \emph{decision problem} is given by an arbitrary compact choice set~$D$ and by
a continuous payoff function~$g\colon D\times \NATURE\to \RR$. Consistent with
the framework of \cite{BlackwellEquivalentComparisonOfExperiments53}, we define
the set of \emph{actions} as the compact convex subspace of~$\RR^\NATURE$ given
by the \emph{closed convex hull}:
\begin{equation}
\ACTIONS = \overline{\textrm{co}} 
\{\bp{g(d,\nature)}_{\nature \in \NATURE}, d\in D\} \subset \RR^{\NATURE}
\eqfinp 
\label{eq:ACTIONS}
\end{equation}
The convexity of~$\ACTIONS$ is justified by allowing the agent to randomize over
actions. 

\subsection*{Duality between actions/payoffs and beliefs}

The scalar product between a vector~$v \in \RR^{\NATURE}$ and a 
signed measure~$\signed \in \RR^{\NATURE}$ is 
\( 
 \proscal{\signed}{v} = 
\sum_{\nature \in \NATURE} \signed_{\nature}v_{\nature} 
\). 
This scalar product induces a duality between payoffs/actions and beliefs.
Such a duality is at the core of a series of works in nonexpected utility theory, such as
\cite{Gilboa-Schmeidler:1989,MacMarRusAmbiguity2006,%
cerreia2011complete}.

Under belief~$\belief \in \BELIEFS$, the decision maker chooses 
a decision~$d\in D$ that maximizes~$\sum_k p_k g(d,k)$, or, equivalently, 
an action~$\action \in \ACTIONS$ that maximizes~$\proscal{p}{a}$, and the corresponding 
\emph{expected payoff} is
\( \max_{\action\in \ACTIONS} \proscal{\belief}{\action} \in \RR \).
We define the \emph{value function}~$\Value_{\ACTIONS}: \BELIEFS \to \RR$ by: 
\begin{equation}
  \Value_{\ACTIONS}(\belief) = 
\max_{\action\in \ACTIONS} \proscal{\belief}{\action} 
\eqsepv \forall \belief \in \BELIEFS \eqfinp 
\label{eq:value_function}
\end{equation}
The value function~$\Value_{\ACTIONS}\colon \BELIEFS \to \RR$ is
convex --- as the supremum of the family of 
affine functions~$\proscal{\cdot}{\action}$ for~$\action\in \ACTIONS$ ---
and continuous --- as its effective domain is the whole convex set~$\BELIEFS$
\cite[p.~175]{Hiriart-Ururty-Lemarechal-I:1993}.

Given a belief~$\belief \in \BELIEFS$, 
we let~$\OptimalActions(\belief) \subset \ACTIONS$ be the 
\emph{set of optimal actions at belief~$\belief$}, given by 
\begin{equation}
\OptimalActions(\belief) = 
\arg\max_{\action' \in \ACTIONS} \proscal{\belief}{\action'}
= \{ \action \in \ACTIONS \mid
\forall \action'\in \ACTIONS \eqsepv 
\proscal{\belief}{\action'} \leq \proscal{\belief}{\action} \} \eqfinp
\label{eq:face_beliefs}
\end{equation}
Geometrically, the set~$\OptimalActions(\belief)$ is the \emph{(exposed) face of~$\ACTIONS$ 
in the direction~$\belief \in \BELIEFS$} (see \eqref{eq:face_signed}
in Appendix for a proper definition).
The set~$\OptimalActions(\belief)$ is nonempty, closed and
convex (as~$\ACTIONS$ is convex and compact).

Conversely, an outside observer can make inferences 
on the agent's beliefs from observed actions.
For an action~$\action\in\ACTIONS$,  
the set~$\RevealedBeliefs(\action)$  of 
\emph{beliefs revealed by action~$\action$} 
is the set of all beliefs for which~$\action$ is an optimal action,
given by:
\begin{equation}
\RevealedBeliefs(\action) = \{ \belief \in \BELIEFS \mid
\forall \action'\in \ACTIONS \eqsepv 
\proscal{\belief}{\action'} \leq \proscal{\belief}{\action} \} \eqfinp
\label{eq:normal_cone_beliefs} 
\end{equation}
Geometrically, the set~$\RevealedBeliefs(\action)$ is the intersection with~$\BELIEFS$ 
of the \emph{normal cone}~$\NORMAL_{\ACTIONS}(\action)$ (see \eqref{eq:normal_cone_signed} for a proper definition).

Obviously, given~$\action\in \ACTIONS$ and $\belief\in \BELIEFS$, 
$\action \in \OptimalActions(\belief)$ iff
$\belief \in \RevealedBeliefs(\action)$, 
as both express that action~$\action$ is optimal under belief~$\belief$.

\subsection*{Information structure}

We follow
\cite{BohnenblustShapleySherman49,BlackwellEquivalentComparisonOfExperiments53}, 
and we describe information through a distribution of posterior beliefs 
that average to the prior belief. 
Hence, given the prior belief~$\priorbelief$, we define an \emph{information structure}
as a random variable~$\va{\beliefbis}$, defined over a probability space
$(\Omega,{\cal F},\PP)$ and with values in~$\BELIEFS$, describing the agent's
posterior beliefs, and such that 
(where~$\EE$ denotes the expectation operator with respect to~$\PP$) 
\begin{equation}
\va{\beliefbis} : (\Omega,{\cal F},\PP) \to \BELIEFS 
\eqsepv  \EE\bc{\va{\beliefbis}} = \priorbelief 
\eqfinp
\label{eq:information_structure}
\end{equation} 

Given the {action set}~$\ACTIONS$ in~\eqref{eq:ACTIONS} and the
information structure~$\va{\beliefbis}$ in~\eqref{eq:information_structure}, 
the \emph{value of information}~$\VoI_{\ACTIONS}\np{\va{\beliefbis}}$ is the
difference between the expected payoff for an agent who receives information
according to~$\va{\beliefbis}$ and one whose prior belief is~$\priorbelief$. It is given by: 
\begin{equation}
  \VoI_{\ACTIONS}\np{\va{\beliefbis}} 
=  \EE \bc{\Value_{\ACTIONS}\np{\va{\beliefbis}}}-\Value_{\ACTIONS}(\priorbelief) 
 \eqfinp
\label{eq:VoI}
\end{equation}

The following example illustrates relations between 
the set~$\ACTIONS$ of actions and the value function~$\Value_\ACTIONS$. 
\begin{example}\label{ex:example1}
Consider two states of nature, $\NATURE=\{1, 2\}$, 
decisions~$D=\{d_1, d_2, d_3, d_4\}$, 
and payoffs given by Table~\ref{Table_of_payoffs}.
\begin{table}[htp]
\begin{center}
\begin{tabular}{|c|c|c|}
\cline{1-3}
	&$k=1$&$k=2$\\
\cline{1-3}
$d_1$&$3$&$0$\\
$d_2$&$2$&$2$\\
$d_3$&$0$&$5/2$\\
$d_4$&$0$&$0$\\
\hline
\end{tabular}
\end{center}
\caption{Table of payoffs}\label{Table_of_payoffs}
\end{table}%
In this case, $\ACTIONS$ is the convex hull of the four 
points~$(3,0)$, $(2,2)$, $(0, 5/2)$ and $(0,0)$. 
The value function~$\Value_\ACTIONS$, expressed as a function of the 
probability~$\belief$ of state~2, is the maximum of the 
following three affine functions: $3(1-\belief)$, 2, and $5\belief/2$. 
Action~$(3,0)$ is optimal for~$\belief\leq1/3$, 
$(2,2)$ is optimal for~$\belief\in[1/3, 4/5]$, and 
$(0,5/2)$ is optimal for~$\belief\geq 4/5$. 
Both the set~$\ACTIONS$ and the function~$\Value_\ACTIONS$ 
are represented in Figure~\ref{figure:example_discrete}.

\begin{figure}[h]
\includegraphics[width=\textwidth]{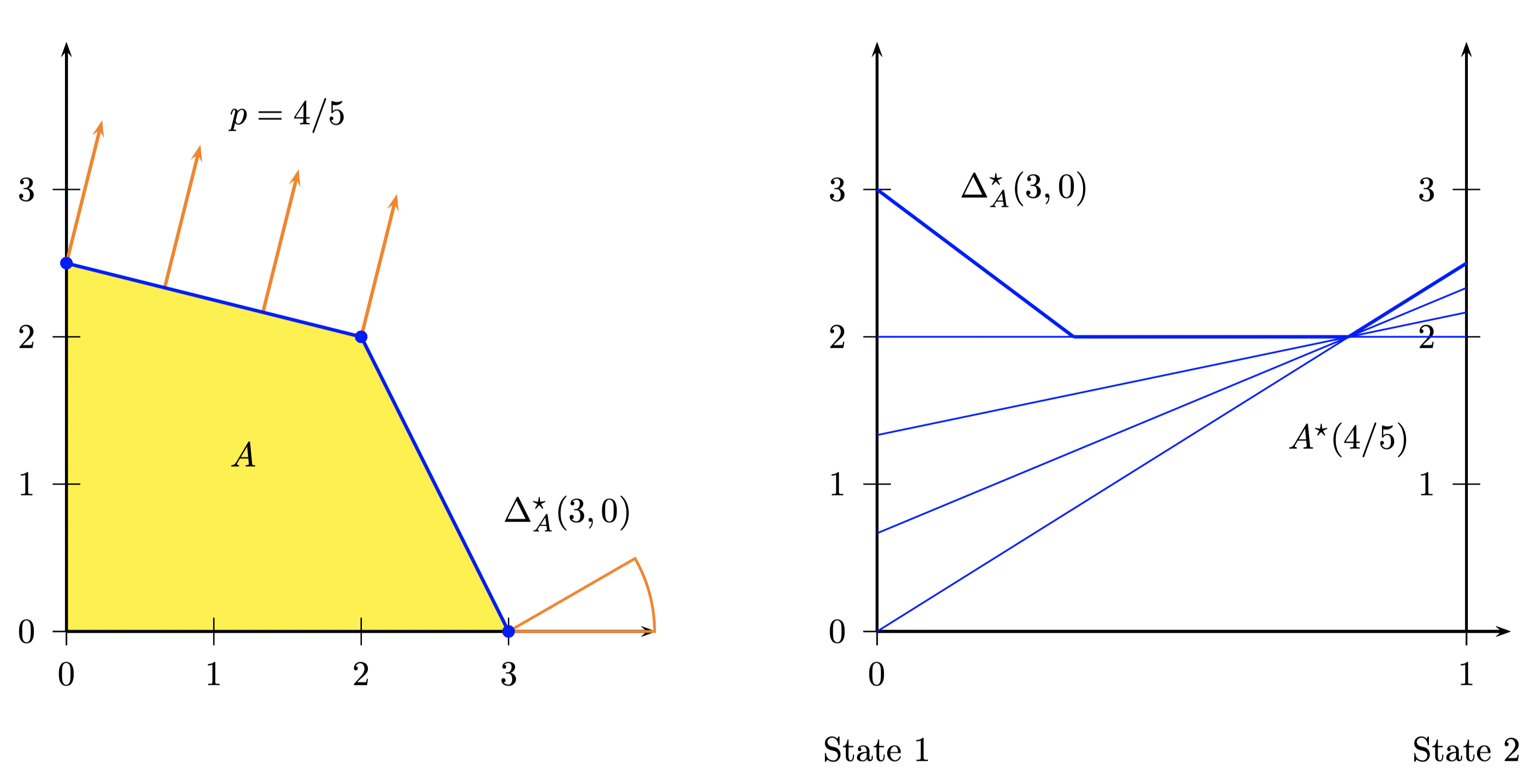}
\caption{The set~$\ACTIONS$ of actions on the left, and 
the value function~$\Value_\ACTIONS$ on the right. 
Each of the four arrows on the left represents an action~$\action$ 
such that~$\belief=4/5$ belongs to 
the set~$\RevealedBeliefs(\action)$ of beliefs revealed by action~$\action$.
On the right side, these four actions (each attached to an arrow)
can be seen as four elements of the subdifferential 
of the value function~$\Value_\ACTIONS$ at~$\belief=4/5$. 
The set~$\RevealedBeliefs(3,0)=[0,1/3]$ can be visualized 
both as the normal cone at~$(3,0)$ on the left side, 
and as the range of values of probabilities~$\belief$ 
for which~$(3,0)$ is optimal on the right.
\label{figure:example_discrete}}
\end{figure}

At~$\belief = 4/5$, the optimal actions are~$(2,2)$, $(0,5/2)$, 
and their convex combinations. 
At this point, the mapping~$\Value_\ACTIONS$ is not differentiable. 
However, its subdifferential --- which can be visualized as the set of 
straight lines that are below~$\Value_\ACTIONS$ and 
tangent to it at~$\belief =4/5$ --- 
is still well defined and corresponds precisely 
to the optimal actions~$\OptimalActions(4/5)$, 
i.e.\,the convex hull of~$\{(2,2),(0,5/2)\}$. 
 
The set~$\RevealedBeliefs(3,0)$ of 
beliefs revealed by action~$(3,0)$ consists of the 
range~$\belief\in[0,1/3]$, and it can be seen on the right side
of Figure~\ref{figure:example_discrete} that, 
for this range of probabilities, the action~$(3,0)$ is optimal 
and that $\Value_\ACTIONS$ is linear and equal to~$3(1-\belief)$. 
\end{example}

\section{On the value of information}
\label{sec:value_info}

In this section, we relate the geometry of the set~$\ACTIONS$ of actions in~\eqref{eq:ACTIONS}
with the behavior of the agent around the prior belief~$\priorbelief$, 
with differentiability properties of the value function~$\Value_\ACTIONS$ in~\eqref{eq:value_function}
at the prior belief~$\priorbelief$, 
and with the value of information~$\VoI_{\ACTIONS}$ in~\eqref{eq:VoI}. 
This approach allows us to derive bounds on the value of information 
that depend on how information influences actions. 

First, in Subsect.~\ref{Valuable_information},
we consider information that does not allow us to eliminate optimal actions. We
introduce the \emph{confidence set} as the set of posterior beliefs at which all
optimal actions at the prior remain optimal. We show that information is
valuable if and only if, with positive probability, it can lead to a posterior
outside this set. Therefore, information is valuable whenever it allows to
eliminate some actions from the set of optimal ones. 

Second, in Subsect.~\ref{Indifferences},
we consider the somewhat opposite case of tie-breaking information. This corresponds to situations in which the agent is indifferent between several actions, and the information allows her to select among them. We show that the value of information can be related to an expected distance between the prior and the posterior, provided that posterior beliefs move in these tie-breaking directions. 

These two first approaches are suitable in finite decision problems where the
value function is piecewise linear. 
In the third approach, in Subsect.~\ref{Flexible_decisions},
we look at situations in which the optimal action is locally unique around the
prior and depends on information  
in a continuous and smooth way. There, we show that the value of information can
essentially be measured as an expected square distance from the prior to the
posterior. This approach is particularly adapted to cases in which the space of
actions is sufficiently rich, and where small changes of beliefs lead to
corresponding small changes of actions.

\subsection{Valuable information}
\label{Valuable_information}

Our first task is to formalize the idea that useful information 
is information that affects optimal choices (quoting \cite{Hir71},
``Information is of value only if it can affect action'').
Since there are potentially several optimal actions at a prior belief~$\priorbelief$ 
and at a posterior~$\belief$, there are in principle many ways 
to formalize this idea. 

We say that a belief~$\belief$ is in the 
\emph{confidence set}~$\ConfidenceSet(\priorbelief)$ 
of prior belief~$\priorbelief$ iff all optimal actions at~$\priorbelief$ 
(those in \( \OptimalActions(\priorbelief) \))
are also optimal at~$\belief$. In other words, we define the 
\emph{confidence set of prior belief~$\priorbelief$} by:
\begin{equation}
\ConfidenceSet(\priorbelief) = 
\bigcap_{\action \in \OptimalActions(\priorbelief)} \RevealedBeliefs(\action) \eqfinp
\label{eq:ConfidenceSet}
\end{equation}
Another way to look at this notion is to consider an 
observer who sees choices by the decision maker: 
$\belief \in \ConfidenceSet(\priorbelief)$ when none of the actions 
chosen by the agent at prior belief~$\priorbelief$ would lead the observer to 
refute the possibility that the agent has belief~$\belief$. 

The notion of a confidence set allows for the characterization 
of valuable information as follows.

\begin{proposition}[Valuable information]
For every information structure~$\va{\beliefbis}$ as
in~\eqref{eq:information_structure},
we have:
\begin{subequations}
\begin{align}
\VoI_{\ACTIONS}\np{\va{\beliefbis}} = 0 
& \iff \exists \action^\star \in \OptimalActions(\priorbelief) \eqsepv 
\action^\star \in \OptimalActions\np{\va{\beliefbis}} \eqsepv \PP-\mbox{a.s.}\\ 
&  \iff \va{\beliefbis} \in \ConfidenceSet(\priorbelief) \eqsepv \PP-\mbox{a.s.}
\end{align}
\end{subequations}
\label{pr:Valuable_information}
\end{proposition}

In Example~\ref{ex:example1}, the confidence set at $\priorbelief = 1/2$ is the
closed interval $[1/3, 4/5]$ (the flat portion of the function to 
the right of Figure~\ref{figure:example_discrete}).
Information is valuable whenever, with some
positive probability, the posterior does not belong to this set. When the
posterior falls in this set with probability one, the
value function averaged at the prior precisely equals the value at prior
belief~$\priorbelief$, hence information has no value.  

It is relatively straightforward to see that if all posteriors remain 
in the confidence set, information is valueless. In fact, 
when this is the case, the same action is optimal for all of the posteriors, 
which means that the agent can play this action,
while ignoring the new information, and obtain the same value. 
The proposition shows that the  converse result also holds: 
the value of information is positive whenever posteriors fall outside 
of the confidence set with some positive probability. 

More can be said about estimates on the value of information. 
To do so, we introduce an $\varepsilon$-neighborhood 
of the confidence set~$\ConfidenceSet(\priorbelief)$. 
For $\varepsilon>0$, let 
\begin{equation}
  {\BELIEFS_{\ACTIONS,\varepsilon}^{\textrm{c}}}(\priorbelief)=
\{\beliefbis \in \BELIEFS \mid 
d\bp{\beliefbis,\ConfidenceSet(\priorbelief)}< \varepsilon\} 
\mtext{ where }
  d\bp{\beliefbis,\ConfidenceSet(\priorbelief)}= 
\inf_{\belief \in \ConfidenceSet(\priorbelief)} \Vert \belief-\beliefbis\Vert .
\label{eq:distance}
\end{equation}
This leads us to a first estimate of the value of information.
\begin{theorem}[Bound on the value of information based on confidence sets]
For every~$\varepsilon>0$, there exist positive 
constants~$\ConstantSupspeciale$  and $\ConstantInfEpsilon$ 
such that, for every information structure~$\va{\beliefbis}$ as in~\eqref{eq:information_structure}:
\begin{equation}
\ConstantSupspeciale 
\EE \bc{ d\bp{\va{\beliefbis},\ConfidenceSet(\priorbelief)} } 
\geq \VoI_{\ACTIONS}\np{\va{\beliefbis} } 
\geq \ConstantInfEpsilon \PP\{ \va{\beliefbis} \not\in 
{\BELIEFS_{\ACTIONS,\varepsilon}^{\textrm{c}}}(\priorbelief) \} 
\eqfinp
\label{eq:Valuable_information}
\end{equation}
\label{th:Valuable_information}
\end{theorem}

The upper bound tells us that the value of information is bounded by (a constant
times) the expected distance from the posterior to the confidence set at the
prior. In particular, it is bounded by the expected distance from the posterior
to the prior itself. The lower bound is a converse result, but in which we need
to replace the confidence set by some $\varepsilon$-neighborhood. It shows us
that the value of information is bounded below by (a constant times) the
probability that the posterior is at least distance~$\varepsilon$ from the
confidence set, 
and, therefore, it is also larger than the expected distance
 from the posterior to this $\varepsilon$-neighborhood of the confidence set. 
Both the lower and upper bounds depend on the confidence set
$\ConfidenceSet(\priorbelief)$ in~\eqref{eq:ConfidenceSet}, which can be computed locally at
prior belief~$\priorbelief$. On the other hand, they apply to all information structures. The
caveat is that the multiplicative constants $\ConstantSupspeciale$ and
$\ConstantInfEpsilon$ in~\eqref{eq:Valuable_information} depend on global, and
not just local, properties of the action set~$\ACTIONS$.

\subsection{Undecided} 
\label{Indifferences}

We now consider situations in which information influences actions the most. 
Those are situations of indifference in which, 
at the prior belief~$\priorbelief$, 
the agent is \emph{undecided} between several optimal actions. 
A small piece of information can then be enough to break this indifference. 
As shown by the following proposition (whose proof we do not give,
as it is well-known in convex analysis
\cite[p.~251]{Hiriart-Ururty-Lemarechal-I:1993}), the value function 
then exhibits a \emph{kink} at prior belief~$\priorbelief$.

\begin{proposition}
The two following conditions are equivalent:
\begin{itemize}
\item the set~$\OptimalActions(\priorbelief)$
of optimal actions at the prior belief~$\priorbelief$ in~\eqref{eq:face_beliefs}
contains more than one element;
\item the value function~$\Value_\ACTIONS$ in~\eqref{eq:value_function} is nondifferentiable 
(in the standard sense) at the prior belief~$\priorbelief$.
\end{itemize}
\label{pr:Indifferences}
\end{proposition}

Cases of indifference are typical of situations 
with a finite number of action choices. 
Coming back to Example~\ref{ex:example1}, 
the agent is undecided for $\priorbelief = 1/2$ and
$\priorbelief = 3/4$: at these priors, 
the agent has several optimal choices, and the value function is
nondifferentiable. 
At all other priors, the optimal choice is unique, and the value function is differentiable. 

At prior beliefs~$\priorbelief$ 
satisfying the conditions of Proposition~\ref{pr:Indifferences},
the convexity gap of the value function~$\Value_\ACTIONS$ 
is maximal in the directions in which it is nondifferentiable. 
This allows us to derive a second bound on the value of information. 
For this purpose,
we call \emph{indifference kernel}~$\IndifferenceKernel(\priorbelief)$ 
at prior belief~$\priorbelief$ the vector space of signed measures
that are orthogonal to all differences of optimal
actions~$\OptimalActions(\priorbelief)$ at~$\priorbelief$, that is, 
\begin{equation}
  \IndifferenceKernel(\priorbelief)
= \left[ \OptimalActions(\priorbelief) 
- \OptimalActions(\priorbelief) \right]^{\perp} \eqfinp
\label{eq:IndifferenceKernel}
\end{equation}
Beliefs in the 
{indifference kernel}~$\IndifferenceKernel(\priorbelief)$ 
do not break any of the ties in~$\OptimalActions(\priorbelief)$, since
\( 
\belief \in  \IndifferenceKernel(\priorbelief) \iff
\proscal{\belief}{\action} =\proscal{\belief}{\action'} \eqsepv
\forall \np{\action, \action'} \in \OptimalActions(\priorbelief)^2
\). 
We note the inclusion
\( 
\ConfidenceSet(\priorbelief)
\subset 
\IndifferenceKernel(\priorbelief) \cap \BELIEFS 
\) 
as every element in the confidence set is necessarily in the indifference kernel 
and in the simplex of probability measures.

Recall that a \emph{seminorm}
on the signed measures~$\SIGNED$ on~$\NATURE$, identified with~$\RR^{\NATURE}$, 
is a mapping  
\( \seminorm{\cdot} : \RR^{\NATURE} \to \RR_+ \) which satisfies the requirements
of a norm, except that the vector subspace \( \{ \signed \in \RR^{\NATURE} \mid
\seminorm{\signed}=0 \} \) --- 
called the \emph{kernel} of the seminorm~$\seminorm{\cdot}$
--- is not necessarily reduced to the null vector.

\begin{theorem}[Bounds on the value of information for the undecided agent]
There exists a positive constant~$\ConstantSupspeciale$ 
and a seminorm~$\Vert \cdot \Vert_{\IndifferenceKernel(\priorbelief)}$ 
with kernel~$\IndifferenceKernel(\priorbelief)$, the 
{indifference kernel} in~\eqref{eq:IndifferenceKernel},
such that, for every information structure~$\va{\beliefbis}$ as in~\eqref{eq:information_structure}:
\begin{equation}
\ConstantSupspeciale \EE \Vert \va{\beliefbis} - \priorbelief \Vert \geq 
\VoI_{\ACTIONS}\np{\va{\beliefbis}} \geq 
\VoI_{\OptimalActions(\priorbelief)}\np{\va{\beliefbis}} \geq 
\EE \Vert \va{\beliefbis} - \priorbelief 
\Vert_{\IndifferenceKernel(\priorbelief)} \eqfinp 
\label{eq:Indifferences}
\end{equation}
\label{th:Indifferences}
\end{theorem}

For $\priorbelief = 1/2$ or $\priorbelief = 3/4$ in Example~\ref{ex:example1},  
Theorem~\ref{th:Indifferences} shows that the value of information for these
priors is bounded above and below by a constant times the norm-1 between the
prior and the posterior.  
Since any small amount of information allows to break the indifference 
between the optimal actions at these priors, information is very valuable. 
 
The lower bound in Theorem~\ref{th:Indifferences}
shows that a lower bound of the value of information 
is the expectation of a seminorm of the distance between the prior belief
 and the posterior belief. To understand the role of the 
kernel~$\IndifferenceKernel(\priorbelief)$ 
of this seminorm, let us first consider the set of beliefs in this set. 
A posterior~$\beliefbis$ is 
in~$\IndifferenceKernel(\priorbelief)=
[\OptimalActions(\priorbelief)-\OptimalActions(\priorbelief)]^\perp$  
if and only if, for any two optimal actions
\( \action, \action' \in \OptimalActions(\priorbelief) \), 
\( \proscal{\beliefbis}{\action} =\proscal{\beliefbis}{\action'} \).
In words, posteriors that do not break any of the ties 
in~$\OptimalActions(\priorbelief)$ might not be valuable to the agent. 
On the other hand, Theorem~\ref{th:Indifferences}
tells us that all other directions --- 
i.e., those that allow at least one of the ties 
in~$\OptimalActions(\priorbelief)$ to be broken --- are valuable to the agent, 
and furthermore, in these directions, the value of information behaves 
like an expected distance from the prior to the posterior. 

The upper bound says that the value of information is bounded by an expected
distance from the prior to the posterior, and the inner inequality states that
the value of information with decision set~$\ACTIONS$ is at least as large as
with action set~$\OptimalActions(\priorbelief)$.  

Note that the bounds on Theorem~\ref{th:Indifferences} rely on the indifference
kernel $\IndifferenceKernel(\priorbelief)$ in~\eqref{eq:IndifferenceKernel}, 
which can be computed directly from
the set $\OptimalActions(\priorbelief)$ by~\eqref{eq:IndifferenceKernel}. 
The multiplicative constant $\ConstantSupspeciale$ in~\eqref{eq:Indifferences},
however, depends on more global properties of the action set~$\ACTIONS$.

\subsection{Flexible} 
\label{Flexible_decisions}

Finally, we consider the case in which there is a unique optimal action for each
belief in the range considered, and this action depends smoothly on the
belief. More precisely, we assume that, around the prior, optimal actions
smoothly depend in a 1-1 way on the belief. This assumption is met when, for
instance, the decision problem faced by the agent is a scoring rule
\cite{BrierScoringRules1950}, or an investment problem
\cite{Arrow71,CabralesGossnerSerranoEntropyAER2013}. 

Our first step is to characterize a class of situations of interest, in which
the agent's optimal action depends smoothly on her belief. 
The following proposition offers three alternative characterizations of these
situations, based 1) on the local behavior of the agent's optimal optimal
choices, 2) on local properties of the geometry of the boundary of the set of
actions, and 3) on local second differentiability properties of the value
function. 
For background on geometric convex analysis, 
the reader can consult~\S\ref{Recalls_on_geometric_convex_analysis}
in the Appendix. 

\begin{proposition}
Suppose that the action set~$\ACTIONS$ in~\eqref{eq:ACTIONS} 
has boundary~$\partial\ACTIONS$ 
which is a~$C^2$ submanifold of~$\RR^{\NATURE}$ 
of dimension~\( | \NATURE | -1 \).
The three following conditions are equivalent:
\begin{enumerate}
\item 
The set-valued mapping 
of optimal actions at the prior belief~$\priorbelief$ in~\eqref{eq:face_beliefs}
\begin{equation}
\OptimalActions\colon  \BELIEFS \rightrightarrows \partial\ACTIONS 
\eqsepv \belief \mapsto \OptimalActions(\belief) 
\label{eq:face_set-valued_mapping}
\end{equation}
is a local diffeomorphism\footnote{%
In particular, the set \( \OptimalActions(\belief) \) is a singleton for all 
\( \belief \in \BELIEFS \), in which case we identify a singleton set 
with its single element. 
} 
at the prior belief~$\priorbelief$;
\label{it:curved_diffeomorphism}
\item 
The set~$\OptimalActions(\priorbelief)$ of optimal actions at 
the prior belief~$\priorbelief$ in~\eqref{eq:face_beliefs} is reduced to a 
singleton at which the curvature of the 
action set~$\ACTIONS$ is positive;
\label{it:curved_face}
\item 
The value function~$\Value_{\ACTIONS}$ in~\eqref{eq:value_function} is twice differentiable at
the prior belief~$\priorbelief$, with positive definite Hessian at~$\priorbelief$.
\label{it:curved_value_function}
\end{enumerate}
In this case, we say that the agent is \emph{flexible at $\priorbelief$.}
\label{pr:Flexible_decisions}
\end{proposition}

\begin{theorem}[Bounds on the VoI for the flexible agent] 
If the agent is flexible at prior belief~$\priorbelief$, then 
there exist positive constants~$\ConstantSup$ and $\ConstantInf$
such that, for every information structure~$\va{\beliefbis}$ as in~\eqref{eq:information_structure}:
\begin{equation}
  \ConstantSup \EE || \va{\beliefbis} - \priorbelief ||^2
\geq \VoI_{\ACTIONS}\np{\va{\beliefbis}} \geq 
\ConstantInf \EE || \va{\beliefbis} - \priorbelief ||^2 
\eqfinp 
\label{eq:Flexible_decisions}
\end{equation}
\label{th:Flexible_decisions}
\end{theorem}
Theorem~\ref{th:Flexible_decisions} shows that, in the case of a flexible agent,
the value of information is essentially given by the expected square distance
between the prior and the posterior, up to some multiplicative constant. One of
the strengths of the theorem is that its assumption that the agent is flexible
is a local one, whereas its conclusion is global, as it applies to all
information structures. On the other hand, the multiplicative
constants~$\ConstantSup$ and $\ConstantInf$ in~\eqref{eq:Flexible_decisions}
themselves depend on the global behavior of the value function, 
and hence cannot be inferred from local properties only.

\section{An insurance example}
\label{sec:insure}

In this example, we study an insurance problem
and illustrate how the results of Sect.~\ref{sec:value_info} apply.
The insuree chooses whether to insure, or not, 
and at which indemnity level to insure if she does. 
The uncertainty is about the level of risk she incurs, 
and she may receive some partial information about it.

\begin{example}
The model is drawn from the classical insurance framework  
(see \cite{Bernoulli1738,EekhoudtGollierSchlesinger2005}).  

An insuree faces the decision of partially or 
fully insuring a good of value~$\varpi$ against the possibility of its total loss. 
Pricing is assumed to be linear, so that, 
for an indemnity~$I$, the insurance company charges
\begin{equation}
P(I) = \alpha I+f \text{ where } \alpha \in ]0,1[ \eqsepv f>0
\eqfinp
\end{equation}
In exchange for the premium~$P(I)$, the
insuree gets compensation of an amount~$I$ from the insurance company in case of
a loss. For the range of wealth~$\wealth$ considered, the insuree's utility
function~$\utility$ is considered to have 
constant absolute risk aversion~$\riskaversion$,
that is, 
\begin{equation}
\utility(\wealth)=1-e^{-\riskaversion\wealth} 
\eqfinp 
\end{equation}
By~\eqref{eq:ACTIONS}, 
the set of \emph{actions} is the closed convex hull
\begin{equation}
\ACTIONS = \overline{\textrm{co}} 
\Ba{ \Bp{ \utility\np{\varpi} , \utility\np{0} } \eqsepv
\Bp{ \utility\bp{-P(I)+\varpi} , \utility\bp{-P(I)+I} } }
\label{eq:ACTIONS_insurance}
\end{equation}
where, by convention, the first coordinate corresponds to no loss 
and the second corresponds to the loss. 

The insuree's subjective perception that a loss may arise 
is~$\belief\in ]0,1[$, probability of loss. The insuree
chooses either not to insure, 
and obtains expected utility
\begin{subequations}
\begin{equation}
\Utility_0\np{\belief}= 
(1-\belief) \utility\np{\varpi} + \belief \utility\np{0} 
=  (1-\belief) \bp{ 1 - e^{-\riskaversion\varpi} } 
\eqfinv 
\end{equation}
or to insure for an indemnity~$I> 0$ that maximizes the expected utility 
\begin{equation}
\Utility\np{\belief,I}=
(1-\belief)\utility\bp{-P(I)+\varpi} 
+ \belief \utility\bp{-P(I)+I} 
=
1-\belief e^{-\riskaversion\bp{-P(I)+I}} -
(1-\belief)e^{-\riskaversion\bp{-P(I)+\varpi}} 
\eqfinp  
\label{eq:EU_EekhoudtGollierSchlesinger2005_b} 
\end{equation}
\label{eq:EU_EekhoudtGollierSchlesinger2005}
\end{subequations}
\end{example}	
The question now becomes whether no insurance
or a positive level of indemnity is chosen. 

\begin{proposition}
There exists a threshold belief $\belief^* \in ]0,1[$ 
and a smooth function \( \hat{I} : [\belief^*,1] \to ]0,+\infty[ \) 
such that
\begin{enumerate}
\item 
for $\belief < \belief^*$, it is optimal not to insure,
\item 
for $\belief = \belief^*$, 
the insuree is indifferent between no insurance and
  insurance at the positive indemnity level~$\hat{I}(\belief^*)$, 
\item 
for~$\belief > \belief^*$, it is optimal 
to insure at the positive indemnity level~$\hat{I}\np{\belief} $.
\end{enumerate}
\label{prop:insurance}
\end{proposition}

\begin{proof} 
It is easy to see that the function 
\( I \in \RR \mapsto \Utility\np{\belief,I} \) 
in~\eqref{eq:EU_EekhoudtGollierSchlesinger2005_b}
is strictly concave with a unique maximum, characterized by 
\( \partial{\Utility}/\partial{I}=0 \), and achieved at 
\begin{equation}
\hat{I}\np{\belief} = \varpi - \frac1\riskaversion 
\ln(\frac{1-\belief}{\belief}\frac{\alpha}{1-\alpha}) 
\eqsepv \forall \belief\in ]0,1[
\eqfinp   
\label{eq:maximizer}
\end{equation}
We denote by~$\hat{\belief}$ the unique $\belief\in ]0,1[$ such that 
\( \hat{I}\np{\belief} > 0 \iff \belief > \hat{\belief} \).
To determine whether no insurance
or a nonnegative level of indemnity is chosen, we introduce
the difference of expected utilities
\begin{equation}
  \delta\np{\belief} = 
\max_{I \geq 0}\Utility\np{\belief,I} - \Utility_0\np{\belief} 
=
\begin{cases}
\Utility\np{\belief,0} - \Utility_0\np{\belief} 
& \text{if } \belief \leq \hat{\belief} 
\eqfinv
\\
  \Utility\bp{\belief,\hat{I}\np{\belief}} - \Utility_0\np{\belief} 
& \text{if } \belief \geq \hat{\belief} 
\eqfinp
\end{cases}
\label{eq:delta}
\end{equation}
We study the behavior of the function~$\delta$ 
when $\belief$ is small and when $\belief$ is close to one.
After computation, we find that, for all $\belief\in [0,1]$ ,
\( U(p,0) - \Utility_0\np{\belief} = - \bp{e^{Rf} - 1 } 
\bp{ \belief + \np{1-\belief} e^{-R\varpi}} < 0 \). 
Therefore, \( \delta\np{\belief} < 0 \) 
for all \( \belief \leq \hat{\belief} \).
On the other hand, when $\belief$ goes to~1, 
\( \delta\np{\belief} \) goes to~1 
because $\Utility_0\np{\belief} \to 0$ and
\( \Utility\bp{\belief,\hat{I}\np{\belief}} =
\np{1-\belief}
\Bp{ 1-e^{-\riskaversion\bp{-P\np{\hat{I}\np{\belief}}+\varpi } } }
+ \belief
\Bp{ 1-e^{-\riskaversion\bp{ -P\np{\hat{I}\np{\belief}} + \hat{I}\np{\belief} }
  } }
= 1 - \np{1-\belief}
\bp{\frac{1-\belief}{\belief}\frac{\alpha}{1-\alpha}}^\alpha
e^{\riskaversion(1-\alpha)\varpi} 
- \belief
\bp{\frac{1-\belief}{\belief}\frac{\alpha}{1-\alpha}}^{1-\alpha}
e^{-\riskaversion(1-\alpha)\varpi}  \to~1 \) (as $\alpha \in ]0,1[$).
As a consequence, we can define 
\( \belief^* = \inf \defset{ \belief\in [0,1] }{ \delta\np{\belief} > 0 } \),
which belongs to \( [\hat{\belief}, 1[ \). Indeed, 
since \( \delta\np{\belief} < 0 \) for \( \belief \leq \hat{\belief} \),
we deduce that \( \belief^* \geq \hat{\belief} \); and
\( \belief^* < 1 \) because \( \delta\np{\belief} \to~1 \) 
when \( \belief \to~1 \). 
We now check that \( \belief^* \) 
and \( \hat{I} \) in~\eqref{eq:maximizer}
satisfy the three assertions of the Proposition.
\medskip

By definition of~\( \belief^* \) and of the function~$\delta$, 
for $\belief < \belief^*$, it is optimal not to insure.

As the function~$\delta$ is continuous, we have 
\( \delta\np{\belief^*}=0 \) and
the insuree is indifferent between no insurance and
  insurance at the positive indemnity level~$\hat{I}(\belief^*)$.

To finish, we will now show that \( \delta\np{\belief} >0 \) when $\belief >
\belief^*$, leading to the conclusion that it is optimal 
to insure at the positive indemnity level~$\hat{I}\np{\belief} $.
Indeed, for $\belief > \belief^*$, we have 
\begin{align*}
  \delta\np{\belief}
&=
\delta\np{\belief}-\delta\np{\belief^*} 
\quad \text{ as \( \delta\np{\belief^*}=0 \) }
\\
&=
\Utility\bp{\belief,\hat{I}\np{\belief}}
  -\Utility\bp{\belief,\hat{I}\np{\belief^*}} 
+ \Utility\bp{\belief,\hat{I}\np{\belief^*}} 
- \Utility_0\np{\belief} 
- \bc{ \Utility\bp{\belief^*,\hat{I}\np{\belief^*}} 
- \Utility_0\np{\belief^*} }
\text{ by~\eqref{eq:delta} }
\\
& >
\Utility\bp{\belief,\hat{I}\np{\belief^*}} 
- \Utility_0\np{\belief} 
-\Utility\bp{\belief^*,\hat{I}\np{\belief^*}} 
+ \Utility_0\np{\belief^*} 
\text{ \qquad as } \Utility\bp{\belief,\hat{I}\np{\belief}}
  -\Utility\bp{\belief,\hat{I}\np{\belief^*}} > 0 
\\
& \text{ by definition of the maximizer~$\hat{I}\np{\belief}$}
\text{ and since \( \hat{I}\np{\belief} > \hat{I}\np{\belief^*} \geq 0 \) 
as $\belief > \belief^* \geq \hat{\belief} $}
\\
&=
(1-\belief) \bc{ \utility\bp{-P\np{\hat{I}\np{\belief^*}}+\varpi} - \utility\np{\varpi} }
+ \belief \bc{ \utility\bp{-P\np{\hat{I}\np{\belief^*}}+\hat{I}\np{\belief^*}} -
  \utility\np{0} }
\\
&-
(1-\belief^*) \bc{ \utility\bp{-P\np{\hat{I}\np{\belief^*}}+\varpi} - \utility\np{\varpi} }
- \belief^* \bc{ \utility\bp{-P\np{\hat{I}\np{\belief^*}}+\hat{I}\np{\belief^*}} -
  \utility\np{0} }
\text{ by~\eqref{eq:EU_EekhoudtGollierSchlesinger2005} }
\\
&=
\np{ \belief - \belief^* } \Bc{ 
\bc{ \utility\bp{-P\np{\hat{I}\np{\belief^*}}+\hat{I}\np{\belief^*}} -
  \utility\np{0} }
+
\bc{ \utility\np{\varpi} 
- \utility\bp{-P\np{\hat{I}\np{\belief^*}}+\varpi} } 
} \geq 0 
\end{align*}
since both terms between inner brackets are increments of the increasing
function~$\utility$, where 
\( -P\np{\hat{I}\np{\belief^*}}+\hat{I}\np{\belief^*} \geq 0 \)
(to be seen below) 
and \( P\np{\hat{I}\np{\belief^*}} \geq 0 \) 
(because \( \hat{I}\np{\belief^*} \geq 0 \)). 
If we had \( -P\np{\hat{I}\np{\belief^*}}+\hat{I}\np{\belief^*}
< 0 \), we would arrive at the contradiction that 
\( 0=\delta\np{\belief^*}=(1-\belief^*) 
\bc{ \utility\bp{-P\np{\hat{I}\np{\belief^*}}+\varpi} - \utility\np{\varpi} }
+ \belief^* \bc{ \utility\bp{-P\np{\hat{I}\np{\belief^*}}+\hat{I}\np{\belief^*}} -
  \utility\np{0} } < 0 \) since both terms between brackets are 
(negative) increments of the increasing
function~$\utility$. 
\end{proof}
\bigskip

Now, we assume that the insuree has access to a small piece of information
concerning her probability of loss. Once informed, she discovers that the
probability $\beliefbis$ of a loss is  either~$\belief-\varepsilon$
or~$\belief+\varepsilon$, where both possibilities are equally likely and
$\varepsilon>0$ is a small positive number. Let~$\Value(\beliefbis)$ be the utility of the
insuree with beliefs~$\beliefbis$, once the optimal policy is chosen: 
\begin{equation}
\Value(\beliefbis) = \max\left\{ \Utility_0(q), \max_{I\geq 0}U(q,I) \right\} 
\eqfinp 
\label{eq:value_function_insurance}
\end{equation}
As $\Value$ is the value function in~\eqref{eq:value_function}, 
the value of information in the decision problem is defined as the expected
utility with the information minus the expected utility absent the information,
as in~\eqref{eq:VoI}: 
\begin{equation}
\VoI(\varepsilon) = \frac12 \Value(\belief+\varepsilon) + 
\frac12\Value(\belief-\varepsilon)- \Value(\belief) 
\eqfinp 
\end{equation}
Note that $\VoI (\varepsilon)$ measures the value of information in terms of
utility; the equivalent measure in monetary terms would be~$-\frac1\riskaversion
\ln(1-\VoI(\varepsilon))$. The following proposition characterizes the value of
a small amount of information, in terms of the agent's optimal insurance
behavior.  

\begin{proposition} 
Depending on the probability of loss~$\belief$, the value of information for
small~$\varepsilon$ behaves as follows: 
\begin{enumerate}
\item 
In the confident case, 
for $\belief< \belief^*$, $\VoI(\varepsilon) = 0$ for small~$\varepsilon$,
\item 
In the undecided case,
for~$\belief = \belief^*$, $\VoI(\varepsilon) \sim C^* \varepsilon$ for a constant~$C^* >0$,
\item 
In the flexible case, 
for~$\belief > \belief^*$, 
$\VoI(\varepsilon) \sim  C\np{\belief} \varepsilon^2$ for a constant~$C\np{\belief}>0$.
\end{enumerate}
\end{proposition}

\begin{proof}
	The confident and undecided cases are immediate consequences of
        Theorems~\ref{th:Valuable_information} and~\ref{th:Indifferences},
        together with Proposition \ref{prop:insurance}. In the flexible case,
        the optimal indemnity level is given by $\hat{I}\np{\belief}>0$, 
and the function \( \hat{I} : ]\belief^*,1] \to ]0,+\infty[ \)
in~\eqref{eq:maximizer}
is differentiable with $\frac{d\hat{I}\np{\belief}}{d\belief}\neq 0$. 
The set of optimal actions $\OptimalActions\np{\belief}$ 
in~\eqref{eq:face_beliefs} is reduced to the single point
\(\OptimalActions\np{\belief} = \)
\( \Bp{ 1-e^{-\riskaversion\bp{-P\bp{\hat{I}\np{\belief}}+\varpi ) } },
1-e^{-\riskaversion\bp{ -P\bp{\hat{I}\np{\belief}} + \hat{I}\np{\belief} } } } \).
As the curve $\belief \in ]\belief^*,1] \mapsto \OptimalActions\np{\belief}$ 
has a derivative that never vanishes, we deduce that it is a local
diffeomorphism (onto its image in~\( \partial \ACTIONS \))
at~$\belief$, and Theorem \ref{th:Flexible_decisions} applies.
\end{proof}
\bigskip

The results of Proposition~\ref{prop:insurance} are intuitive. First, a small
piece of information is valueless if the agent is not buying insurance. For such
agents, a small bit of information does not affect behavior, as even bad news is
not enough to trigger insurance purchase. For an undecided agent who is
indifferent between no insurance and insurance at a positive indemnity
level~$I(\belief^*)$, a small piece of information is enough to break the
indifference and significantly influences her behavior; this is the
situation in which information is the most valuable. Finally, for an agent who
takes a positive level of indemnity, information may affect the level of
indemnity chosen. But, because the change of indemnity level is itself of order
$\varepsilon$, and the indemnity level~$I(\belief^*)$ is $\varepsilon$-optimal
at the posterior, the value of information is a second order in $\varepsilon$. 

Figure~\ref{fig:insurance} represents the set~$\ACTIONS$ of
actions~\eqref{eq:ACTIONS_insurance} to the left, and the corresponding
value function~$\Value=\Value_{\ACTIONS}$ in~\eqref{eq:value_function_insurance}
to the right. 
In the representation of~$\ACTIONS$, the horizontal axis
corresponds to the payoff without loss, and the vertical axis to the payoff in
case of a loss. The circled dot to the right corresponds to the choice of no
insurance; it maximizes payoff in case of no loss. The thick curve represents
the set of payoffs that are achieved by different coverage levels. Finally,
$\ACTIONS$ is the convex hull of this set of points; it appears under the dashed
contour. As seen on the value function graph, for low values of the
probability~$\belief$ of loss, the value
function is linear as the insuree chooses not to purchase insurance. At $\belief^*$
(which is approximately 0.334), the value function exhibits a kink, and the agent
is indifferent between no insurance and a positive indemnity level.
Finally, for larger values of $\belief$, the value function~$\Value$ 
is twice continuously differentiable with a positive second derivative,
and the optimal insurance level is a smooth and positive function of the
insuree's belief.
\begin{figure}
\begin{center}
\includegraphics[scale=0.4]{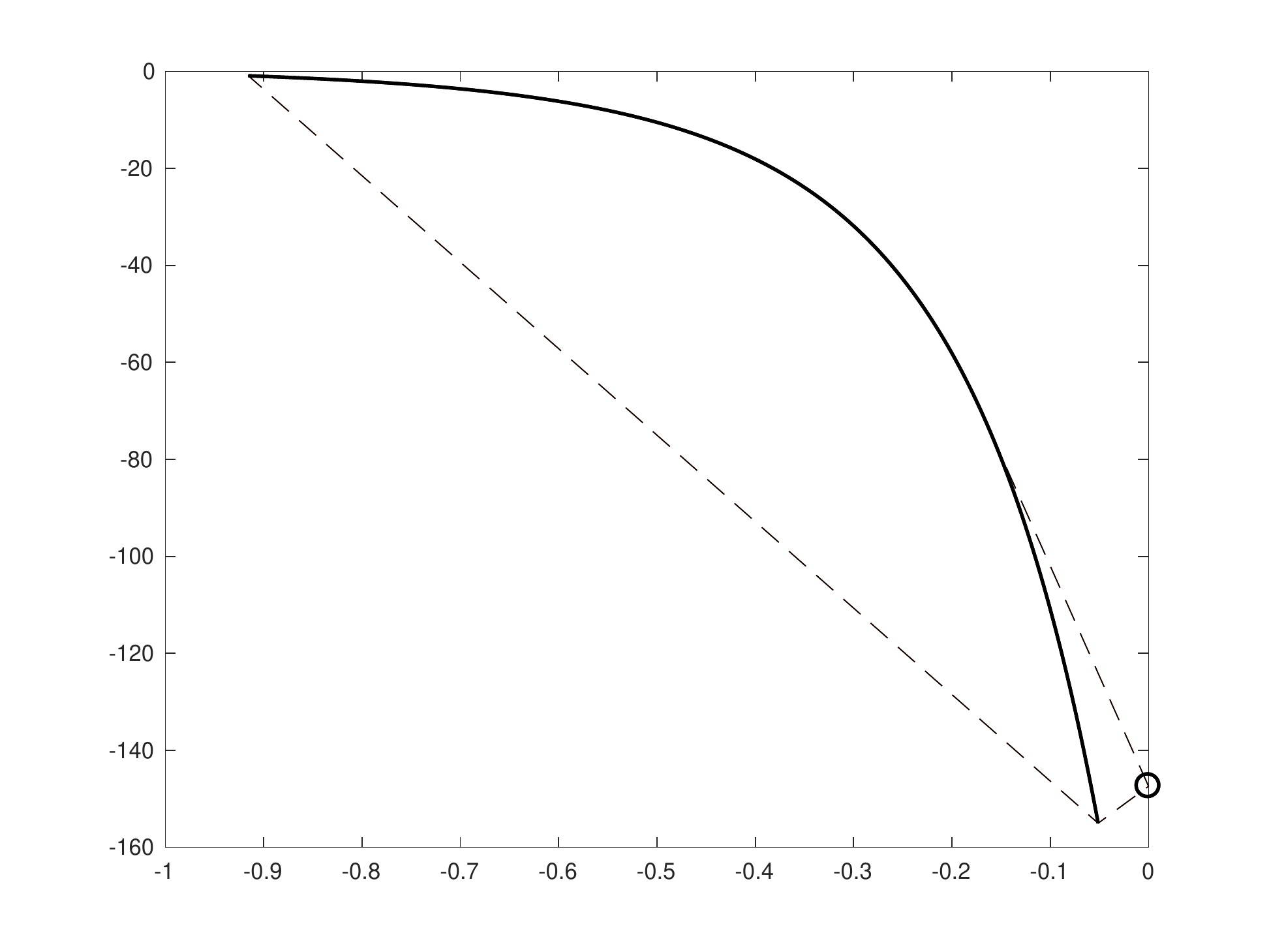}
\includegraphics[scale=0.4]{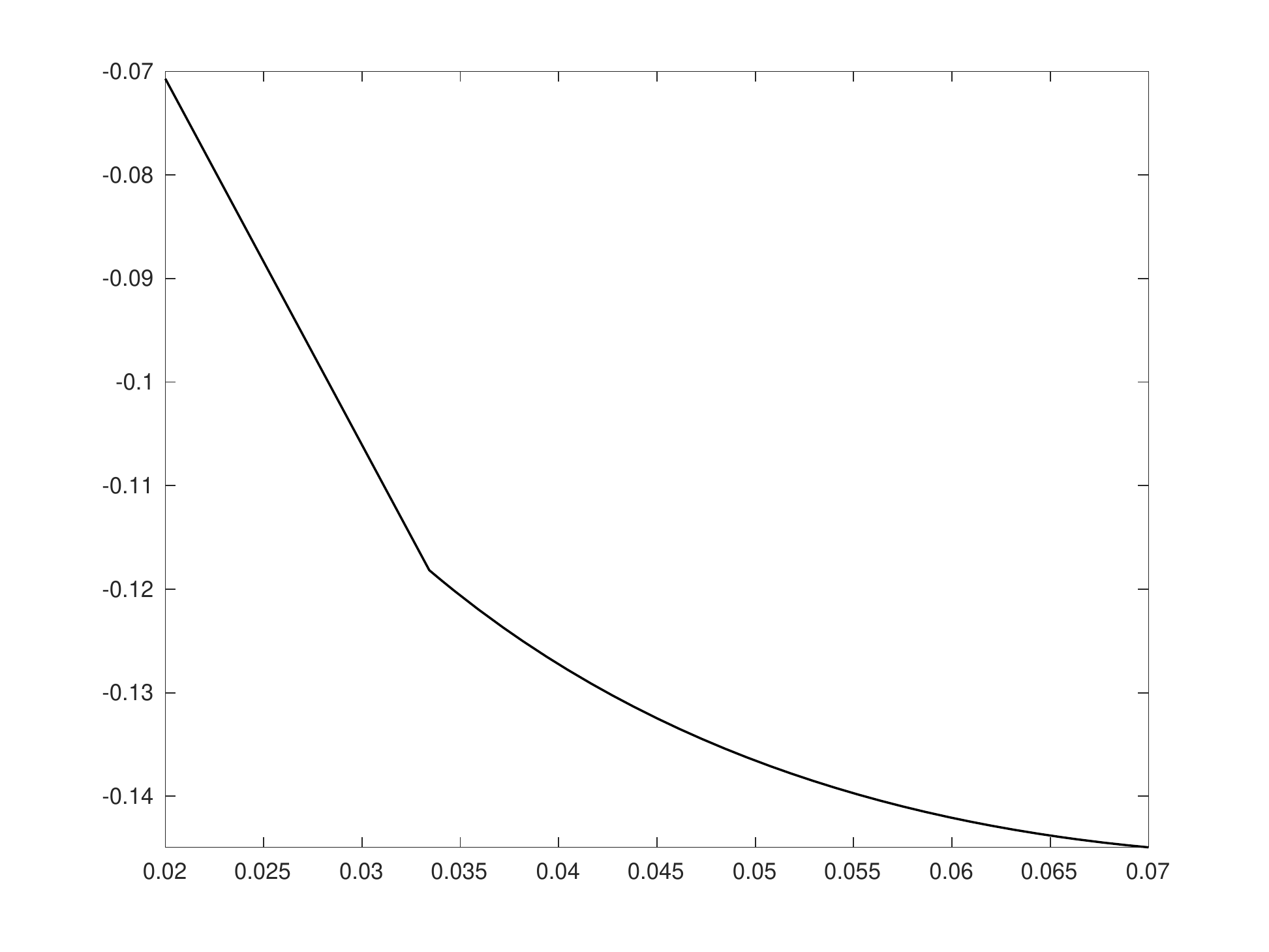}
\end{center}
\caption{The action set~$\ACTIONS$ on the left and the corresponding value
  function~$\Value=\Value_{\ACTIONS}$ in~\eqref{eq:value_function_insurance} for
  the insurance example on the right. Parameter values are $\alpha=0.08$,
$f=10$, $\varpi=1000$, $\riskaversion=10$.}\label{fig:insurance}
\end{figure}
	

\section{The marginal value of information}
\label{sec:marginal}

The question of the marginal value of information 
is studied in~\cite{RadnerStiglitzNonConcavity1984}.
They provide joint conditions on a parameterized family of information structures
together with a decision problem such that, when the agent is close to receiving
no information at all, the marginal value of information is null. Their result
was subsequently generalized in~\cite{Chade_ShleeJET2002-another} and
\cite{delara2007tight}, where are provided joint conditions on parameterized
information and a decision problem leading to zero marginal value of information.

In this Section, we show how our bounds on the value of information,
obtained in Sect.~\ref{sec:value_info}, apply to the
marginal value of information. 
In Subsect.~\ref{Model_and_first_result}, we provide separate conditions on
the decision problem and on the family of parameterized information structures
that result in a null value of information.  
We then examine, in Subsect.~\ref{Examples}, several
parameterized families of information structures and rely on our main results to
study how the marginal value of information varies depending on the decision
problem faced.

\subsection{Model and first result}
\label{Model_and_first_result}

Let \( \np{ \va{\beliefbis}^\theta }_{\theta > 0 } \)
be a family of information structures as in~\eqref{eq:information_structure}.
As in \cite{RadnerStiglitzNonConcavity1984}, 
we are interested in the so-called \emph{marginal value of information}: 
\begin{equation}
V^+ = \limsup_{\theta\to 0} \frac1\theta 
\VoI_{\ACTIONS}\np{\va{\beliefbis}^\theta}
\eqfinp 
\label{eq:marginal_VoI}
\end{equation}

The following proposition is a straightforward consequence of
Theorems~\ref{th:Valuable_information} 
and \ref{th:Flexible_decisions}.

\begin{proposition}
	Assume that
        \begin{itemize}
		\item 
either $\EE \bc{ d\bp{\va{\beliefbis}^\theta,\ConfidenceSet(\priorbelief)} }
= o(\theta)$,
		\item 
or the decision maker is \emph{flexible} at prior belief~$\priorbelief$ and 
$\EE \Vert  \va{\beliefbis}^\theta-\priorbelief\Vert^2 = o({\theta})$.
        \end{itemize}
	Then the marginal value of information $V^+ = 0$. 
\label{pr:marginal_value_of_information} 
\end{proposition}
The first condition is met automatically if 
$\EE \Vert  \va{\beliefbis}^\theta-\priorbelief\Vert = o({\theta})$. It is also
met if, for instance, $\ConfidenceSet(\priorbelief)$ has a nonempty interior, and
posteriors converge to the prior almost surely. 
\medskip

We now discuss how our approach in 
Proposition~\ref{pr:marginal_value_of_information}
compares with the literature. In \cite{RadnerStiglitzNonConcavity1984}, one finds
joint conditions on the parameterized information structure 
\( \np{ \va{\beliefbis}^\theta }_{\theta > 0 } \)
and the decision problem at hand, 
leading to $V^+=0$.
The second case in Proposition~\ref{pr:marginal_value_of_information}, 
when the decision maker is flexible, 
compares with the original Radner-Stiglitz assumptions 
for the smoothness part, 
but not for the uniqueness of optimal actions.
Indeed, Assumption (A0) in \cite{RadnerStiglitzNonConcavity1984} 
does not require that \( \OptimalActions(\va{\beliefbis}^\theta) \) 
be a singleton, for all $\theta$. 

The authors of~\cite{Chade_ShleeJET2002-another} 
make a step towards disentangling 
conditions on the parameterized information structure 
\( \np{ \va{\beliefbis}^\theta }_{\theta > 0 } \)
from conditions on the decision problem
that lead to a null marginal value of information.  
However, like \cite{RadnerStiglitzNonConcavity1984}, 
they make an assumption on how the optimal action varies with information, 
which makes the comparison with 
Proposition~\ref{pr:marginal_value_of_information} delicate.
In addition, \cite{Chade_ShleeJET2002-another} provide sufficient
conditions for $V^+=0$ that bear on the conditional distribution of the signal 
knowing the state of nature. 
Our approach focuses on the posterior 
conditional distribution of the state of nature knowing the signal.

The authors of~\cite{delara2007tight} 
provide separate conditions on the parameterized information structure 
\( \np{ \va{\beliefbis}^\theta }_{\theta > 0 } \)
and the decision problem (represented by the action set~$\ACTIONS$)
that lead to $V^+=0$.
Their condition ``IIDV=0'' is that 
\( \limsup_{\theta\to 0} \frac1\theta 
\EE \Vert  \va{\beliefbis}^\theta-\priorbelief\Vert = 0 \), or, equivalently,
$\EE \Vert  \va{\beliefbis}^\theta-\priorbelief\Vert = o({\theta})$, 
which implies the first item of 
Proposition~\ref{pr:marginal_value_of_information}.
Thus, this latter proposition implies the main result of \cite{delara2007tight}.

\subsection{Examples}	
\label{Examples}

Here, we study the marginal value of information for several typical
parameterized information structures. In the first example, information consists
on the observation of a Brownian motion with known variance and a drift that
depends on the state of nature. In the second example, information consists of
the observation of a Poisson process whose probability of success depends on the
state of nature. In these two well studied families in the learning literature,
the natural parameterization of information is the length of the interval of time
during which observation takes place. In the third example, 
the agent observes a binary signal and the marginal value of information depends
on the asymptotic informativeness of these signals close to the situation without information. 

In all three following examples we assume binary states of nature, $\NATURE =
\{0,1\}$, and (by a slight abuse of notation) 
the prior belief on the state being~$1$ is 
denoted~$\priorbelief \in ]0,1[$. We
follow the conditions in Sect.~\ref{sec:value_info} 
under which we established bounds on the value of
information, and label as: 
``confident'' the case in which $\priorbelief$ lies in the interior of
the confidence set~$\Delta_A^c(\priorbelief)$ 
(in this case, $\Delta_A^c(\priorbelief)$ is a closed nonempty interval 
$\bc{\belief_l, \belief_h}$ by Proposition~\ref{pr:ConfidenceSet_ter}, 
and the value function is linear on this range);
``undecided'' the case in which the decision problem
faced by the decision maker is such that there is indifference between two
actions at prior belief~$\priorbelief$; 
``flexible'' the case in which the optimal action is a smooth function of the
belief in a neighborhood of prior belief~$\priorbelief$.

Our aim is to develop estimates of the marginal value of information~$V^+$
in~\eqref{eq:marginal_VoI}. 
There are three possibilities: it can be infinite, null, 
or positive and finite. We denote these three cases by $V^+ =\infty$, 
$V^+ =0$ and $V^+ \simeq1$ respectively.


\begin{example}[Brownian motion] 
Frameworks in which agents observe a Brownian motion with known volatility and unknown drift include 
\cite{bergemann1997market,KellerRadyExperimentation99,BolHar99},
as well as reputation models like 
\cite{FaingoldSannikovReputationContinuousTime2011}. 

Assume the agent observes the realization of a Brownian motion 
with variance 1 and drift~$\nature\in\{0,1\}$, namely 
\( d\va{Z}_t = \nature dt + d\va{B}_t \), 
for a small interval of time~$\theta >0$. 
If we let~$\va{\beliefbis}^t$ be the posterior belief at time~$t$, 
it is well-known\footnote{See for instance Lemma 1 in \cite{BolHar99} 
or Lemma 2 in \cite{FaingoldSannikovReputationContinuousTime2011}.} 
that $\va{\beliefbis}^t$ follows a diffusion process of the form
\( d \va{\beliefbis}^t = \va{\beliefbis}^t(1-\va{\beliefbis}^t) d \va{w}_t \),
where $\va{w}$ is a standard Browian process.
Thus, for small values of $\theta$, we have the estimates
\begin{equation*}
\EE \Vert \va{\beliefbis}^\theta-\belief \Vert \sim \sqrt{\theta} \eqsepv 
\EE \Vert \va{\beliefbis}^\theta-\belief \Vert^2 \sim \theta \eqfinp
\end{equation*}
It follows from
Theorems~\ref{th:Valuable_information}-\ref{th:Flexible_decisions} 
that the marginal value of information is characterized, 
depending on the decision problem, as:
\begin{enumerate}
\item In the confident case, $V^+ = 0$,
\item In the undecided case, $V^+ = \infty$,
\item In the flexible case, $V^+\simeq 1$.
\end{enumerate}
\end{example}


\begin{example}[Poisson learning]
An important class of models of strategic experimentation 
(see \cite{kellerradycripps2005strategic}) are those in which 
the agent's observations are driven by a Poisson process of unknown intensity. 
Assume the agent observes, during a small interval of time~$\theta>0$, 
 a Poisson process with intensity~$\rho_\nature$, $\nature\in\{0,1\}$,
where~$\rho_1 > \rho_0 > 0 $. 
The probability of two successes is negligible compared to the probability of
one success (of order~$\theta^2$ compared to~$\theta$). A success leads to a
posterior that converges from below, as~$\theta\to0$, to
\begin{equation*}
\beliefbis^+ = 
\frac{\priorbelief\rho_1}{\priorbelief\rho_1+(1-\priorbelief)\rho_0} > 
\priorbelief \eqfinv
\end{equation*}
and happens with probability of order $\sim\theta$. 
In the absence of success, the posterior belief converges to
the prior belief~$\priorbelief$ 
as~$\theta\to 0$. 
As we have seen that the confidence set~$\Delta_A^c(\priorbelief)$ 
is a closed interval \( \bc{\belief_l, \belief_h} \), 
we note that $\EE\bc{ d\bp{\va{\beliefbis}^\theta, \Delta_A^c(\priorbelief)}}
\sim\theta$ if $q^+ > p_h$, 
and $\EE\bc{ d\bp{\va{\beliefbis}^\theta, \Delta_A^c(\priorbelief)} }
=o(\theta)$ otherwise. 
This implies:
\begin{enumerate}
\item In the confident case, 
\begin{enumerate}
\item~$V^+\simeq1$ if~$\beliefbis^+>\belief_h$, 
\item~$V^+\simeq0$ if~$\beliefbis^+\leq \belief_h$.
\end{enumerate}
\end{enumerate}
\smallskip

We also have the estimates
\begin{equation*}
\EE \Vert \va{\beliefbis}^\theta-\belief \Vert \sim \theta \eqsepv
\EE \Vert \va{\beliefbis}^\theta-\belief \Vert^2 \sim \theta 
\eqfinv 
\end{equation*}
which imply the following estimates on the marginal value of information:
\begin{enumerate}\setcounter{enumi}{1}
\item In the undecided case, $V^+\simeq1$,
\item In the flexible case, ~$V^+\simeq1$.
\end{enumerate}
\end{example}


\begin{example}[Equally likely signals] 
Here, we consider binary and equally likely signals, which lead to a ``split''
of beliefs around the prior belief~$\priorbelief$. Depending on the precision of
these signals as a function of~$\theta$,  
the posterior beliefs are~$p \pm \theta^\alpha$ 
for a certain parameter~$\alpha>0$
(lower values of~$\alpha$ correspond to more spread out beliefs around the
prior, hence to more accurate information).  
In this case we easily compute
\begin{equation*}
\EE \Vert \va{\beliefbis}^\theta-\belief \Vert = \theta^\alpha \eqsepv
\EE \Vert \va{\beliefbis}^\theta-\belief \Vert^2 = \theta^{2\alpha} 
\eqfinv 
\end{equation*} and we observe that 
$\EE\bc{d(\va{\beliefbis}^\theta, \Delta_A^c(\priorbelief)}=0$ for $\theta$ small enough. 
Here again, the marginal value of information is deduced from 
Theorems~\ref{th:Valuable_information}--\ref{th:Flexible_decisions}:
\begin{enumerate}
\item In the confident case, $V^+ = 0$,
\item In the undecided case,
\begin{enumerate}
\item~$V^+ = \infty$ if~$\alpha<1$, 
\item~$V^+\simeq1$  if~$\alpha=1$,
\item~$V^+ = 0$ if~$\alpha>1$, 
\end{enumerate}
\item In the flexible case, 
\begin{enumerate}
\item~$V^+ = \infty$ if~$\alpha<1/2$, 
\item~$V^+\simeq1$  if~$\alpha=1/2$,
\item~$V^+ = 0$ if~$\alpha>1/2$.
\end{enumerate}
\end{enumerate}
\end{example}


Table~\ref{tb:summarises_the_marginal_value_of_information} 
summarizes the marginal value of information in all of our examples.

\begin{table}[h]
\begin{center}\begin{tabular}{||l||c|c|c||}
\hline
Marginal value of information~$V^+$ &confident&undecided&flexible
\\
\hline\hline
Brownian&0&$\infty$&1\\
Poisson learning & 0 or 1&1&1\\
Equally likely signals, $\alpha<1/2$&0&$\infty$&$\infty$\\
Equally likely signals, $\alpha=1/2$&0&$\infty$&$1$\\
Equally likely signals, $1/2<\alpha<1$&0&$\infty$&$0$\\
Equally likely signals, $\alpha=1$&0&$1$&$0$\\
Equally likely signals, $\alpha>1$&0&$0$&$0$\\
\hline
\end{tabular}\end{center}
\caption{Marginal value of information in the different examples. 
The value~1 represents a positive and finite marginal value of information.
\label{tb:summarises_the_marginal_value_of_information}}
\end{table}

In all cases except one, the marginal value of information is completely
determined by the local behavior of the value function around the prior. For the
Poisson case, the marginal value of information is 0 or positive, depending on
whether the observation of a success is sufficient to lead to a decision
reversal.  

The marginal value of information is always weakly lower in the flexible case
than in the undecided case, and weakly higher in the undecided case than in
other cases. In the confident case, the marginal value of information is null,
except in the Poisson case with~$\beliefbis^+> \belief_h $. This is driven by
the fact that, in all other cases, posteriors are, with high probability, too
close to the prior to lead to a decision reversal. In the undecided situation,
the marginal value of information is always positive or infinite, except for
sufficiently uninformative binary signals ($\alpha>1$). Finally, in the flexible
case --- the most representative of decision problems with a continuum of
actions --- the value of information is positive or infinite, except with quite
uninformative binary signals ($\alpha > 1/2$).

\section{Related literature}
\label{sec:related_literature}

The value of information in decision problems is a well-studied question in
economics and in statistics. The central work in this area is
\cite{BlackwellEquivalentComparisonOfExperiments53}, which defines a source of
information~$\alpha$ as \emph{more informative} than another, $\beta$, whenever
all agents, independently of their preferences and decision problems faced,
weakly prefer~$\alpha$ to~$\beta$. Blackwell \cite{BlackwellEquivalentComparisonOfExperiments53} characterizes precisely this relationship in the following terms:~$\alpha$ is more informative than~$\beta$ if and only if information from~$\beta$ can be obtained as a garbling of the information from~$\alpha$. 

The requirement that all agents agree on their preferences between two
statistical experiments is a strong one. It implies that this ranking is
incomplete, as many such pairs of experiments cannot be ranked according to this
ordering. Some authors have considered subclasses of decision problems in order
to obtain rankings that are more complete than Blackwell's. For instance,
\cite{Lehmann}, \cite{PersicoAuctions2000} and \cite{atheylevin2017value}
restrict attention to families of decision problems that 
generate monotone decision rules. Focusing on investment decision problems,
\cite{CabralesGossnerSerranoEntropyAER2013} obtains and characterizes a complete
ranking of information sources based on a uniform criterion;
\cite{CabralesGossnerSerrano2017} uses a duality approach to characterize the
value of an information purchase that consists of an information structure with
a price attached to~it.  

The present work departs from this literature in the sense that we focus on the
value of information for a given agent, instead of trying to measure the value
of information independently of the agent. Papers
\cite{GilboaLehrerValInfoJME91} and \cite{AzrieliLehrer} characterize the
possible preferences for information that any agent can have, letting the
decision problem vary and the agent's preferences vary. 

The question of marginal value of information is studied in
\cite{RadnerStiglitzNonConcavity1984,Chade_ShleeJET2002-another,delara2007tight}. They
consider parameterized information structures, and derive general conditions on
the couple consisting of the information structures and the decision problem under
which the marginal value of information close to no information is zero. Our
work contributes to this question by allowing us to derive estimates on the
value of information based on separate conditions on the decision problem and on
the information structure. This is the approach we have taken in
Sect.~\ref{sec:marginal}. 
Our contribution considerably opens the spectrum of possibilities
for the marginal value of information, by giving conditions under which 
it can be infinite, null, or positive and finite.  
\bigskip

\textbf{Acknowledgements.}
Olivier Gossner acknowledges support from the French National Research Agency
(ANR), ``Investissements d'Avenir'' (ANR-11-IDEX-0003/LabEx
Ecodec/ANR-11-LABX-0047), 
and thanks Rafael Veiel for his excellent research assistantship.
The Authors thank the Editor and two Referees for their insightful comments,
that have helped improve the manuscript.

\bibliographystyle{siamplain}

\appendix 

\section{Appendix}

\subsection{Revisiting the model of Sect.~\ref{sec:model}
with convex analysis tools}

We revisit the model in Sect.~\ref{sec:model}
with convex analysis tools
to prepare the proofs in Sect.~\ref{Proofs_of_the_results}. 
We recall that $\ACTIONS \subset \RR^{\NATURE}$ in~\eqref{eq:ACTIONS} is a
nonempty, convex and compact subset of~$\RR^{\NATURE}$, 
called the \emph{action set}, and that we identify the set~$\SIGNED$ of signed
measures on~$\NATURE$ with~$\RR^{\NATURE}$.

\paragraph{Support function}

The \emph{support function}~$\sigma_\ACTIONS$ of the action set~$\ACTIONS$ 
is defined by
\begin{equation}
\sigma_\ACTIONS(\signed)=
\sup_{\action\in \ACTIONS} \proscal{\signed}{\action} 
\eqsepv \forall \signed \in \SIGNED \eqfinp 
\label{eq:support_function}
\end{equation}
The value function~$\Value_{\ACTIONS} : \BELIEFS \to \RR$ 
in~\eqref{eq:value_function} is the restriction of 
$\sigma_\ACTIONS$ 
to probability distributions~$\BELIEFS=\Delta(\NATURE)\subset\SIGNED$:
\begin{equation}
  \Value_{\ACTIONS}(\belief) = \sigma_\ACTIONS(\belief) \eqsepv 
\forall \belief \in \BELIEFS \eqfinp
\label{eq:value_function_support_function}
\end{equation}
It is well-known that~$\sigma_\ACTIONS$ is convex
(as the supremum of the family of 
linear maps~$\proscal{\cdot}{\action}$ for~$\action\in \ACTIONS$).
As the {action set}~$\ACTIONS$ is compact, $\sigma_\ACTIONS(\signed)$
takes finite values, hence its effective domain is~$\SIGNED$, hence~$\sigma_\ACTIONS$ is 
continuous.

\paragraph{(Exposed) face}

For any signed measure~$\signed \in \SIGNED$, 
we let 
\begin{equation}
\FACE_{\ACTIONS}(\signed) = 
\arg\max_{\action' \in \ACTIONS} \proscal{\signed}{\action'}
= \{ \action \in \ACTIONS \mid
\forall \action'\in \ACTIONS \eqsepv 
\proscal{\signed}{\action'} \leq \proscal{\signed}{\action} \}
\subset  \ACTIONS
\label{eq:face_signed}
\end{equation}
be the set of maximizers of~$\action \mapsto \proscal{\signed}{\action}$
over~$\ACTIONS$.
We call~$\FACE_{\ACTIONS}(\signed)$ the \emph{(exposed) face of~$\ACTIONS$ 
in the direction~$\signed \in \SIGNED$}.
As the {action set}~$\ACTIONS$ is convex and compact,
the face~$\FACE_{\ACTIONS}(\signed)$ of~$\ACTIONS$ 
in the direction~$\signed$ is nonempty, for any~$\signed \in \SIGNED$, 
and the face is a subset 
of the \emph{boundary}~$\partial\ACTIONS$ of~$\ACTIONS$:
\( 
  \FACE_{\ACTIONS}(\signed) \subset \partial\ACTIONS \eqsepv
\forall \signed \in \SIGNED  
\). 
We will use the following property:
for any nonempty convex set $\Convex \subset \Primal$ 
and $ \dual \in \Dual$ such that 
$\FACE_{\Convex}(\dual) \neq \emptyset$, we have 
\begin{equation}
  {\sigma_{\Convex}}(\dual') - {\sigma_{\Convex}}(\dual) 
\geq \sigma_{\FACE_{\Convex}(\dual)} (\dual'-\dual) 
\geq \proscal{\dual'-\dual}{\primal'} 
\eqsepv \forall \dual' \in \Dual \eqsepv \forall \primal' \in \Convex
\eqfinp 
\label{eq:property_subdifferential_of_the_support_function}
\end{equation}
The {set~$\OptimalActions(\belief)$ 
of optimal actions under belief~$\belief$} in~\eqref{eq:face_beliefs}
coincides with the {(exposed) face~$\FACE_{\ACTIONS}(\belief)$ of~$\ACTIONS$ 
in the direction~$\belief$} in~\eqref{eq:face_signed}: 
\begin{equation}
\OptimalActions(\belief) = \FACE_{\ACTIONS}(\belief)  \eqsepv
\forall \belief \in \BELIEFS \eqfinp 
\label{eq:face_beliefs_signed}
\end{equation}

\paragraph{Normal cone} 

For any payoff vector~$\action$ in~$\ACTIONS$, we define 
\begin{equation}
\NORMAL_{\ACTIONS}(\action) = \{  \signed \in \SIGNED \mid
\forall \action'\in \ACTIONS \eqsepv 
\proscal{\signed}{\action'} \leq \proscal{\signed}{\action} \} 
\subset \SIGNED \eqfinp
\label{eq:normal_cone_signed}
\end{equation}
We call~$\NORMAL_{\ACTIONS}(\action)$ the \emph{normal cone} 
to the closed convex set~${\ACTIONS}$ at~$\action \in {\ACTIONS}$.
Notice that~$\NORMAL_{\ACTIONS}(\action)$ is made of signed measures in~$\SIGNED$,
that are not necessarily beliefs. 
The set~$\RevealedBeliefs(\action)$ of beliefs 
compatible with optimal action~$\action$ in~\eqref{eq:normal_cone_beliefs}
is related to the {normal cone}~$\NORMAL_{\ACTIONS}(\action)$ at~$\action$ 
 in~\eqref{eq:normal_cone_signed} by: 
\begin{equation}
\RevealedBeliefs(\action) = \NORMAL_{\ACTIONS}(\action)
\cap \BELIEFS \eqsepv \forall \action \in \ACTIONS \eqfinp 
\label{eq:normal_cone_beliefs_signed} 
\end{equation}

\paragraph{Conjugate subsets of actions and beliefs}

Exposed face~$\FACE_{\ACTIONS}$ and normal cone~$\NORMAL_{\ACTIONS}$
are conjugate as follows:
\begin{equation}
 \signed \in \SIGNED \mtext{ and } \action \in \FACE_{\ACTIONS}(\signed) 
\iff 
\action \in \ACTIONS \mtext{ and }
\signed \in \NORMAL_{\ACTIONS}(\action) \eqfinp 
\label{eq:exposed_face_and_normal_cone_are_conjugate}
\end{equation}

\subsection{Background on geometric convex analysis}
\label{Recalls_on_geometric_convex_analysis}

A nonempty, convex and compact set~$\ACTIONS \subset \RR^{\NATURE}$ is called
a \emph{convex body} of~$\RR^{\NATURE}$ \cite[p.~8]{Schneider:2014}.

\paragraph{Regular points and smooth bodies}

We say that a point~$\action \in \ACTIONS$ 
is \emph{smooth} or \emph{regular} \cite[p.~83]{Schneider:2014} if 
the {normal cone}~$\NORMAL_{\ACTIONS}(\action)$ in~\eqref{eq:face_signed}
is reduced to a half-line.
The \emph{set of regular points} is denoted by \( \mathrm{reg}\np{\ACTIONS} \):
\begin{equation}
\action \in \mathrm{reg}\np{\ACTIONS} \iff \exists \signed \in \SIGNED \eqsepv
\signed \not = 0  \eqsepv \NORMAL_{\ACTIONS}(\action)=\RR_+ \signed \eqfinp 
\label{eq:set_of_regular_points}
\end{equation}
Notice that a regular point~$\action$ necessarily 
belongs to the boundary~$\partial\ACTIONS$ of~$\ACTIONS$:
\( \mathrm{reg}\np{\ACTIONS} \subset \partial\ACTIONS \).
The body~$\ACTIONS$ is said to be \emph{smooth} if all boundary points
of~$\ACTIONS$ are regular (\( \mathrm{reg}\np{\ACTIONS}=\partial\ACTIONS \)); 
in that case, it can be shown 
that its boundary~$\partial\ACTIONS$ 
is a~$C^1$ submanifold of~$\RR^{\NATURE}$ \cite[Theorem~2.2.4, p.~83]{Schneider:2014}.

\paragraph{Spherical image map of~$\ACTIONS$}

We denote by~$\sphere= \{ \signed \in \SIGNED \eqsepv \norm{\signed}=1 \} $ 
the unit sphere of the signed measures~$\SIGNED$
on~$\NATURE$
(identified with~$\RR^{\NATURE}$ with its canonical scalar product).
By~\eqref{eq:set_of_regular_points}, we have that 
\( 
\action \in \mathrm{reg}\np{\ACTIONS} \iff \exists ! \signed \in \sphere \eqsepv
\NORMAL_{\ACTIONS}(\action)=\RR_+ \signed 
\). 
If a point~$\action \in \ACTIONS$ is {regular}, 
the unique outer normal unitary vector to~$\ACTIONS$ at~$\action$ 
is denoted by~$\normal_{\ACTIONS}(\action)$, so that
\( \NORMAL_{\ACTIONS}(\action) = \RR_+ \normal_{\ACTIONS}(\action) \).
The mapping 
\begin{equation}
\normal_{\ACTIONS}: \mathrm{reg}\np{\ACTIONS} \to \sphere 
\eqsepv \mtext{where } \mathrm{reg}\np{\ACTIONS} \subset \partial\ACTIONS
\eqfinv  
\label{eq:spherical_image_map}
\end{equation}
is called the \emph{spherical image map of~$\ACTIONS$},
or the \emph{Gauss map}, and is continuous \cite[p.~88]{Schneider:2014}.
We have 
\begin{equation}
  \action \in \mathrm{reg}\np{\ACTIONS} \Rightarrow 
\NORMAL_{\ACTIONS}(\action) = \RR_+ \normal_{\ACTIONS}(\action) 
\mtext{ where } \normal_{\ACTIONS}(\action) \in \sphere \eqfinp
\label{eq:spherical_image_map_characterization}
\end{equation}

\paragraph{Reverse spherical image map of~$\ACTIONS$}

We say that a unit signed measure~$\signed \in \sphere$ is
\emph{regular} \cite[p.~87]{Schneider:2014} if 
the (exposed) face~\( \FACE_{\ACTIONS}(\signed) \)
of~$\ACTIONS$ in the direction~$\signed$,
as defined in~\eqref{eq:face_signed}, is reduced to a singleton.
The \emph{set of regular unit signed measures} 
is denoted by \( \mathrm{regn}\np{\ACTIONS} \):
\begin{equation}
\signed \in \mathrm{regn}\np{\ACTIONS} \iff \signed \in \sphere \mtext{ and }
\exists ! \action \in \ACTIONS
\eqsepv \FACE_{\ACTIONS}(\signed) = \{ \action  \} \eqfinp 
\label{eq:set_of_regular_signed_measures}
\end{equation}
For a regular unit signed measure~$\signed \in \sphere$, we denote by 
\( \face_{\ACTIONS}(\signed) \) the unique element of~\(
\FACE_{\ACTIONS}(\signed) \), so that
\( \FACE_{\ACTIONS}(\signed) = \{ \face_{\ACTIONS}(\signed) \} \).
The mapping 
\begin{equation}
\face_{\ACTIONS} : \mathrm{regn}\np{\ACTIONS} \to \partial\ACTIONS
\eqsepv  \mtext{where } \mathrm{regn}\np{\ACTIONS} \subset \sphere
\eqfinv  
\label{eq:reverse_spherical_image_map}
\end{equation}
is called the \emph{reverse spherical image map of~$\ACTIONS$}, and is
continuous \cite[p.~88]{Schneider:2014}.
We have 
\begin{equation}
 \signed \in \mathrm{regn}\np{\ACTIONS} \Rightarrow 
\FACE_{\ACTIONS}(\signed) = \{ \face_{\ACTIONS}(\signed) \} \eqfinp
\label{eq:reverse_spherical_image_map_characterization}
\end{equation}

\paragraph{Bodies with~$C^2$ surface}

\begin{proposition}[Schneider 2014, p.\,113]
If the body~$\ACTIONS$ has boundary~$\partial\ACTIONS$ 
which is a~$C^2$ submanifold of~$\RR^{\NATURE}$, then 
i) all points~$\action \in \partial\ACTIONS$ are regular
 (\( \mathrm{reg}\np{\ACTIONS}=\partial\ACTIONS \)), 
ii) 
the spherical image map~\( \normal_{\ACTIONS} \) 
in~\eqref{eq:spherical_image_map} is defined 
over the whole boundary~$\partial\ACTIONS$ and is of class~$C^1$,
iii) 
the spherical image map~\( \normal_{\ACTIONS} \) 
has the reverse spherical image map~$\face_{\ACTIONS}$ 
in~\eqref{eq:spherical_image_map} as right inverse,
that is, 
\( 
\normal_{\ACTIONS} \circ \face_{\ACTIONS} =
\mathrm{Id}_{\mathrm{regn}\np{\ACTIONS}} 
\). 
\label{pr:normal_circ_face}
\end{proposition}

\begin{proof}
The first two items can be found in \cite[p.~113]{Schneider:2014}.
Now, we prove that \( \normal_{\ACTIONS} \circ \face_{\ACTIONS} =
  \mathrm{Id}_{\mathrm{regn}\np{\ACTIONS}} \).
As \( \face_{\ACTIONS} : \mathrm{regn}\np{\ACTIONS} \to \partial\ACTIONS \) 
by~\eqref{eq:reverse_spherical_image_map}, and as
\( \normal_{\ACTIONS}: \partial\ACTIONS \to \sphere \) 
by~\eqref{eq:spherical_image_map} since \(
\mathrm{reg}\np{\ACTIONS}=\partial\ACTIONS \), 
the mapping \( \normal_{\ACTIONS} \circ \face_{\ACTIONS} : 
\mathrm{regn}\np{\ACTIONS} \to \sphere \) is well defined. 
Let \( \signed \in \mathrm{regn}\np{\ACTIONS} \). 
By~\eqref{eq:reverse_spherical_image_map_characterization}, we have that 
\( \FACE_{\ACTIONS}(\signed) = \{ \face_{\ACTIONS}(\signed) \} \)
and by~\eqref{eq:spherical_image_map_characterization}, we have that 
\( \NORMAL_{\ACTIONS}\bp{\face_{\ACTIONS}(\signed)} = 
\RR_+ \normal_{\ACTIONS}\bp{\face_{\ACTIONS}(\signed)} \).
From~\eqref{eq:exposed_face_and_normal_cone_are_conjugate} ---
stating that exposed face and normal cone are conjugate --- we deduce that 
\( \signed \in \RR_+ \normal_{\ACTIONS}(\face_{\ACTIONS}(\signed)) \).
As  \( \signed \in \sphere \), we conclude that 
 \( \signed = \normal_{\ACTIONS}\bp{\face_{\ACTIONS}(\signed)} \) 
by~\eqref{eq:spherical_image_map}.
\end{proof}

\paragraph{Weingarten map}

Let \( \action \in \mathrm{reg}\np{\ACTIONS} \) be a regular point,
as in~\eqref{eq:set_of_regular_points},
such that the spherical image map~\( \normal_{\ACTIONS} \) 
in~\eqref{eq:spherical_image_map} is differentiable at~$\action$,
with differential denoted by~$T_{\action}\normal_{\ACTIONS}$.
The \emph{Weingarten map} \cite[p.~113]{Schneider:2014}
\( 
T_{\action}\normal_{\ACTIONS} : T_{\action}\partial\ACTIONS \to
T_{\normal_{\ACTIONS}(\action)} \sphere
\) 
linearly maps the tangent space~$T_{\action}\partial\ACTIONS$ 
of the boundary~$\partial\ACTIONS$
at point~$\action$ into the tangent space~$T_{\normal_{\ACTIONS}(\action)} \sphere$ 
of the sphere~$\sphere$ at~$\normal_{\ACTIONS}(\action)$.
The eigenvalues of the Weingarten map at~$\action$
are called the \emph{principal curvatures} of~$\ACTIONS$ at~$\action$ 
\cite[p.~114]{Schneider:2014};
they are nonnegative \cite[p.~115]{Schneider:2014}.
By definition, the body~$\ACTIONS$ has \emph{positive curvature}
 at~$\action$ if all principal curvatures  at~$\action$
are positive or, equivalently, if the {Weingarten map}
is of maximal rank  at~$\action$ \cite[p.~115]{Schneider:2014}.

\paragraph{Reverse Weingarten map}

Let \( \signed \in \mathrm{regn}\np{\ACTIONS} \) be a 
regular unit signed measure such that the 
reverse spherical image map~$\face_{\ACTIONS}$ 
in~\eqref{eq:reverse_spherical_image_map} is differentiable at~$\signed$,
with differential denoted by~$T_{\signed}\face_{\ACTIONS}$. 
The \emph{reverse Weingarten map}
\begin{equation}
T_{\signed}\face_{\ACTIONS} : T_{\signed} \sphere \to 
T_{\face_{\ACTIONS}\np{\signed}}\partial\ACTIONS 
\label{eq:reverse_Weingarten_map}
\end{equation}
 maps the tangent space~$T_{\signed} \sphere$ 
of the sphere~$\sphere$ at~$\signed$ 
into the tangent space~$T_{\face_{\ACTIONS}\np{\signed}}\partial\ACTIONS$
of the boundary~$\partial\ACTIONS$
at point~$\face_{\ACTIONS}\np{\signed}$.
The eigenvalues of the reverse Weingarten map at~$\signed$ 
are called the \emph{principal radii of curvature} 
of~$\ACTIONS$ at~$\signed$.

\subsection{Proofs of the results in Sect.~\ref{sec:value_info}}
\label{Proofs_of_the_results}

Using the relations~\eqref{eq:face_beliefs_signed}
and \eqref{eq:normal_cone_beliefs_signed}, we express the proofs 
of the results in Sect.~\ref{sec:value_info}
in terms of the sets \( \FACE_{\ACTIONS}(\belief) \) in~\eqref{eq:ACTIONS}
and \( \NORMAL_{\ACTIONS}(\action) \) in~\eqref{eq:normal_cone_signed}
(in the set~$\SIGNED$ of signed measures),
instead of \( \OptimalActions(\belief) \)  in~\eqref{eq:face_beliefs}
and \( \RevealedBeliefs(\action) \) in~\eqref{eq:normal_cone_beliefs} 
(in the set~$\BELIEFS$ of probability measures).
\smallskip

\paragraph{Value of information} 

We have seen in~\eqref{eq:value_function_support_function}
that the value function~$\Value_{\ACTIONS} : \BELIEFS \to \RR$ 
in~\eqref{eq:value_function} is the restriction of the support 
function~$\sigma_\ACTIONS$ to beliefs in~$\BELIEFS$. 
By definition~\eqref{eq:VoI} of the value of information, we deduce that,
for any information structure~$\va{\beliefbis}$  as
in~\eqref{eq:information_structure}, 
we have:
\begin{equation}
  \VoI_{\ACTIONS}\np{\va{\beliefbis}} = 
\EE \left[ \sigma_\ACTIONS(\va{\beliefbis})-\sigma_\ACTIONS(\priorbelief) \right]
\eqfinp 
\label{eq:VoI_bis}
\end{equation}

\begin{lemma}
Let us introduce, for all \( \beliefbis \in \BELIEFS \), 
  \begin{subequations}
    \begin{align}
 \varphi_{\ACTIONS}^+(\beliefbis) 
& =
\sigma_\ACTIONS(\beliefbis)-\sigma_\ACTIONS(\priorbelief) 
+ \sigma_{-\OptimalActions(\priorbelief) }\np{\beliefbis-\priorbelief} \eqfinv \\
\varphi_{\ACTIONS}^-(\beliefbis) 
& =
\sigma_\ACTIONS(\beliefbis)-\sigma_\ACTIONS(\priorbelief) 
- \sigma_{\OptimalActions(\priorbelief) }\np{\beliefbis-\priorbelief} 
\eqfinp 
    \end{align}
\label{eq:varphi_ACTIONS}
  \end{subequations}
Then, for any information structure~$\va{\beliefbis}$ and 
for any \( \action \in \ACTIONS \), we have that
  \begin{subequations}
    \begin{align}
\EE \Bc{ \varphi_{\ACTIONS}^+(\va{\beliefbis}) }
& = \EE \Bc{ \sigma_\ACTIONS(\va{\beliefbis})-\sigma_\ACTIONS(\priorbelief) 
+ \sigma_{-\OptimalActions(\priorbelief) }\np{\va{\beliefbis}-\priorbelief} } \\
& \geq 
\VoI_{\ACTIONS}\np{\va{\beliefbis}} = 
\EE \left[ \sigma_\ACTIONS(\va{\beliefbis})-\sigma_\ACTIONS(\priorbelief) -
\proscal{\va{\beliefbis}-\priorbelief}{\action} \right] 
 \label{eq:VoI_ter} \\
& \geq 
\EE \Bc{ \sigma_\ACTIONS(\va{\beliefbis})-\sigma_\ACTIONS(\priorbelief) 
- \sigma_{\OptimalActions(\priorbelief) }\np{\va{\beliefbis}-\priorbelief} }
= \EE \Bc{ \varphi_{\ACTIONS}^-(\va{\beliefbis}) }
\eqfinp
    \end{align}
\label{eq:VoI_bounds}
  \end{subequations}
\end{lemma}

\begin{proof}
By~\eqref{eq:varphi_ACTIONS}, we have, 
for all \( \beliefbis \in \BELIEFS \), 
  \begin{subequations}
    \begin{align}
 \varphi_{\ACTIONS}^+(\beliefbis) 
& =
\sigma_\ACTIONS(\beliefbis)-\sigma_\ACTIONS(\priorbelief) 
+ \sigma_{-\OptimalActions(\priorbelief) }\np{\beliefbis-\priorbelief} \\
& = 
\sup_{ \action \in \OptimalActions(\priorbelief) } 
\Bp{ \sigma_\ACTIONS(\beliefbis)-\sigma_\ACTIONS(\priorbelief) -
\proscal{\beliefbis-\priorbelief}{\action} } \\
& \geq 
\sigma_\ACTIONS(\beliefbis)-\sigma_\ACTIONS(\priorbelief) -
\proscal{\beliefbis-\priorbelief}{\action}  \eqsepv 
\forall \action \in \OptimalActions(\priorbelief) \\
& \geq 
\inf_{ \action \in \OptimalActions(\priorbelief) } 
\Bp{ \sigma_\ACTIONS(\beliefbis)-\sigma_\ACTIONS(\priorbelief) -
\proscal{\beliefbis-\priorbelief}{\action} } 
\label{eq:varphi_bounds_inf} \\
& =
\sigma_\ACTIONS(\beliefbis)-\sigma_\ACTIONS(\priorbelief) 
- \sigma_{\OptimalActions(\priorbelief) }\np{\beliefbis-\priorbelief} 
=  \varphi_{\ACTIONS}^-(\beliefbis) \eqfinp 
\label{eq:varphi_bounds_-}
    \end{align}
  \end{subequations}
By taking the expectation, we obtain~\eqref{eq:VoI_bounds},
using~\eqref{eq:VoI_bis} and the property that 
\( \EE \left[ \va{\beliefbis} - \priorbelief \right] =0 \)
in~\eqref{eq:information_structure}.
\end{proof}

\paragraph{Confidence set and indifference kernel} 

We start by providing characterizations of 
the confidence set $\ConfidenceSet(\priorbelief)$ in~\eqref{eq:ConfidenceSet}
and of the {indifference kernel}~$\IndifferenceKernel(\priorbelief)$ 
in~\eqref{eq:IndifferenceKernel}, in terms of 
\( \FACE_{\ACTIONS}(\belief) \) in~\eqref{eq:face_signed}
and \( \NORMAL_{\ACTIONS}(\action) \) in~\eqref{eq:normal_cone_signed}.

\begin{proposition}
\quad
  \begin{enumerate}
  \item 
The {confidence set}~$\ConfidenceSet(\priorbelief)$ 
of~\eqref{eq:ConfidenceSet} is the nonempty closed and convex set
 \begin{equation}
\ConfidenceSet(\priorbelief) =
\bigcap_{ \action \in \OptimalActions(\priorbelief) } \RevealedBeliefs(\action) =
\bigcap_{ \action \in \FACE_{\ACTIONS}(\priorbelief) } \NORMAL_{\ACTIONS}(\action) 
\cap \BELIEFS
\eqfinp
\label{eq:ConfidenceSet_bis}
  \end{equation}
\item
Let \( \belief \in \BELIEFS \). We have that 
\begin{subequations}
  \begin{align}
\belief \in \ConfidenceSet(\priorbelief) &
\iff 
\FACE_{\ACTIONS}(\priorbelief) \subset\FACE_{\ACTIONS}(\belief) 
\label{eq:ConfidenceSet_ter_a} \\
& \iff 
\sigma_\ACTIONS(\belief)-\sigma_\ACTIONS(\priorbelief) -
\proscal{\belief-\priorbelief}{\action} = 0 \eqsepv
\forall \action \in \FACE_{\ACTIONS}(\priorbelief) 
\label{eq:ConfidenceSet_ter_b}  \\
& \iff 
\sigma_\ACTIONS(\belief)-\sigma_\ACTIONS(\priorbelief) 
+ \sigma_{-\OptimalActions(\belief) }\np{\belief-\priorbelief} = 0
\eqfinp 
  \end{align}
\label{eq:ConfidenceSet_ter}
\end{subequations}
\item 
The {indifference kernel}~$\IndifferenceKernel(\priorbelief)$ 
of~\eqref{eq:IndifferenceKernel} is the vector subspace
\begin{equation*}
  \IndifferenceKernel(\priorbelief)
= \left[ \FACE_{\ACTIONS}(\priorbelief)
- \FACE_{\ACTIONS}(\priorbelief) \right]^{\perp} 
= \left[ \OptimalActions(\priorbelief) 
- \OptimalActions(\priorbelief) \right]^{\perp} 
= \bigcap_{ \action \in \FACE_{\ACTIONS}(\priorbelief) } 
\NORMAL_{\FACE_{\ACTIONS}(\priorbelief)}(\action) 
\eqfinp
\end{equation*}
  \end{enumerate}
\label{pr:ConfidenceSet_ter}
\end{proposition}

\begin{proof}
\quad 
  \begin{enumerate}
  \item 
Express~\eqref{eq:ConfidenceSet} using~\eqref{eq:normal_cone_beliefs_signed}.
\item 
We prove the three equivalences in~\eqref{eq:ConfidenceSet_ter}.
  \begin{enumerate}
  \item
Let \( \belief \in \BELIEFS \). 
Using the property~\eqref{eq:exposed_face_and_normal_cone_are_conjugate}
that exposed face~$\FACE_{\ACTIONS}$ and normal cone~$\NORMAL_{\ACTIONS}$
are conjugate,
\begin{align*}
\textrm{ we obtain: }
\belief \in \ConfidenceSet(\priorbelief) 
\iff & \belief \in \bigcap_{ \action \in \FACE_{\ACTIONS}(\belief) } 
\NORMAL_{\ACTIONS}(\action) 
\mtext{ by~\eqref{eq:ConfidenceSet_bis} } \\
\iff & 
\action \in \FACE_{\ACTIONS}(\belief) \eqsepv 
\forall \action \in \FACE_{\ACTIONS}(\priorbelief) 
\mtext{ by~\eqref{eq:exposed_face_and_normal_cone_are_conjugate} } 
\iff 
\FACE_{\ACTIONS}(\priorbelief) \subset\FACE_{\ACTIONS}(\belief) \eqfinp
\end{align*}    
\item 
Let \( \belief \in \BELIEFS \). We have that 
\begin{align*}
& 
\sigma_\ACTIONS(\belief)-\sigma_\ACTIONS(\priorbelief) -
\proscal{\belief-\priorbelief}{\action} = 0 \eqsepv
\forall \action \in \FACE_{\ACTIONS}(\priorbelief) \\
\iff &
 \sigma_\ACTIONS(\belief) = \proscal{\belief}{\action} \eqsepv
\forall \action \in \FACE_{\ACTIONS}(\priorbelief)
\intertext{because \( \sigma_\ACTIONS(\priorbelief)=
\proscal{ \priorbelief }{\action} \) for any 
\( \action \in \FACE_{\ACTIONS}(\priorbelief) \),
since $\FACE_{\ACTIONS}(\priorbelief)$ is the 
set~$\OptimalActions(\belief)$ 
of optimal actions under prior belief~$\priorbelief$
by~\eqref{eq:face_beliefs} and~\eqref{eq:face_signed}}
\iff & \belief \in \bigcap_{ \action \in \FACE_{\ACTIONS}(\priorbelief) } 
\NORMAL_{\ACTIONS}(\action) 
\tag{by definition~\eqref{eq:normal_cone_signed} 
of $\NORMAL_{\ACTIONS}(\action) $}
\\
\iff & \belief \in \bigcap_{ \action \in \FACE_{\ACTIONS}(\priorbelief) } 
\NORMAL_{\ACTIONS}(\action) \cap \BELIEFS = \ConfidenceSet(\priorbelief) 
\mtext{ by~\eqref{eq:ConfidenceSet_bis}. } 
\nonumber 
\end{align*}    
%
\item 
For any \( \action \in \ACTIONS \), we define the function 
\begin{equation}
  \varphi_{\action}(\beliefbis)=
\sigma_\ACTIONS(\beliefbis)-\sigma_\ACTIONS(\priorbelief) -
\proscal{\beliefbis-\priorbelief}{\action}  \eqsepv 
\forall \beliefbis \in \BELIEFS \eqfinp 
\label{eq:shifted_value_function}
\end{equation}
By~\eqref{eq:property_subdifferential_of_the_support_function}
and~\eqref{eq:ConfidenceSet_ter_b}, we have that 
%
\begin{subequations}
  \begin{align}
 \forall \action \in \FACE_{\ACTIONS}(\priorbelief) \eqsepv 
 \forall \beliefbis \in \BELIEFS \eqsepv & 
\varphi_{\action}(\beliefbis) \geq 0 \eqfinv 
\label{eq:shifted_value_function_geq0} \\
 \forall \action \in \FACE_{\ACTIONS}(\priorbelief) \eqsepv 
 \forall \beliefbis \in \ConfidenceSet(\priorbelief) \eqsepv & 
\varphi_{\action}(\beliefbis) = 0 \eqfinp 
\label{eq:shifted_value_function_=0} 
  \end{align}
\end{subequations}
Let \( \belief \in \BELIEFS \). 
Using~\eqref{eq:shifted_value_function_geq0}, we deduce 
from~\eqref{eq:ConfidenceSet_ter_b} 
and from the compacity of~$\FACE_{\ACTIONS}(\priorbelief)$ that
\( 
\belief \in \ConfidenceSet(\priorbelief) 
\iff 
\inf_{ \action \in \FACE_{\ACTIONS}(\priorbelief) } \Bp{
\sigma_\ACTIONS(\belief)-\sigma_\ACTIONS(\priorbelief) -
\proscal{\belief-\priorbelief}{\action} } = 0 
\). 
We conclude with~\eqref{eq:varphi_bounds_inf}--\eqref{eq:varphi_bounds_-}.
  \end{enumerate}
\item 
Express~\eqref{eq:IndifferenceKernel} 
using~\eqref{eq:face_beliefs_signed}.
Then, use the definition 
of \( \NORMAL_{\FACE_{\ACTIONS}(\priorbelief)}(\action) \)
in~\eqref{eq:normal_cone_signed}.
  \end{enumerate}
This ends the proof. 
\end{proof}

\subsubsection{Valuable information}

\begin{proof}[Proof of Proposition~\ref{pr:Valuable_information}]

Let \( \action \in \FACE_{\ACTIONS}(\priorbelief) \) and
$\va{\beliefbis}$ be an information structure 
as in~\eqref{eq:information_structure}. 
We have that 
\begin{subequations}
\begin{align*}
\VoI_{\ACTIONS}\np{\va{\beliefbis}} = 0 \iff &  
\EE \left[ \sigma_\ACTIONS(\va{\beliefbis})-\sigma_\ACTIONS(\priorbelief) \right]
=0 \mtext{ by~\eqref{eq:VoI_bis} } \\
 \iff &  
\EE \left[ \sigma_\ACTIONS(\va{\beliefbis})-\sigma_\ACTIONS(\priorbelief) 
- \proscal{ \va{\beliefbis} - \priorbelief }{\action} \right] =0 \eqsepv  
\mtext{ as } \EE \left[ \va{\beliefbis} - \priorbelief \right] =0 \\
 \iff &  
\sigma_\ACTIONS(\va{\beliefbis})-\sigma_\ACTIONS(\priorbelief) 
- \proscal{ \va{\beliefbis} - \priorbelief }{\action} =0 
\eqsepv \PP-\mathrm{a.s.} 
\tag{because \( \sigma_\ACTIONS(\va{\beliefbis})-\sigma_\ACTIONS(\priorbelief) 
- \proscal{ \va{\beliefbis} - \priorbelief }{\action} \geq 0 \)
by~\eqref{eq:property_subdifferential_of_the_support_function} 
since \( \action \in \FACE_{\ACTIONS}(\priorbelief) \)}
\\
 \iff &  
\sigma_\ACTIONS(\va{\beliefbis}) = \proscal{ \va{\beliefbis}}{\action}
\eqsepv \PP-\mathrm{a.s.}
\tag{because \( \sigma_\ACTIONS(\priorbelief)=
\proscal{ \priorbelief }{\action} \) since 
\( \action \in \FACE_{\ACTIONS}(\priorbelief) \)}
\\
 \iff & \PP \left\{\action \in \FACE_{\ACTIONS}(\va{\beliefbis})  \right\} = 1 \\
 \iff & \PP \left\{ \proscal{ \va{\beliefbis} }{\action'-\action} \leq 0 
\eqsepv
\forall \action' \in \ACTIONS \right\} = 1 \eqfinp
\end{align*}
\end{subequations}
\begin{subequations}
Let \( \FACE \subset \FACE_{\ACTIONS}(\priorbelief) \) be a dense subset of
the compact~\( \FACE_{\ACTIONS}(\priorbelief) \) of~$\RR^{\NATURE}$.
We immediately get from the last equality that 
\( \VoI_{\ACTIONS}\np{\va{\beliefbis}} = 0 \Rightarrow \) 
\( \PP \left\{ \proscal{ \va{\beliefbis} }{\action'-\action} \leq 0 \eqsepv
\forall \action' \in \ACTIONS \eqsepv
\forall \action \in \FACE \right\} = 1 \). 
As the set \( \{ \action \in \FACE_{\ACTIONS}(\priorbelief) \mid
\proscal{ \va{\beliefbis} }{\action'-\action} \leq 0 \eqsepv
\forall \action' \in \ACTIONS \} \) is closed (for any outcome
in the underlying sample space~$\Omega$), we get that 
\( \left\{ \proscal{ \va{\beliefbis} }{\action'-\action} \leq 0 \eqsepv
\forall \action' \in \ACTIONS \eqsepv
\forall \action \in \FACE \right\} \subset 
\left\{ \proscal{ \va{\beliefbis} }{\action'-\action} \leq 0 \eqsepv
\forall \action' \in \ACTIONS \eqsepv
\forall \action \in \overline{\FACE} \right\} \). 
We deduce from the last equality that 
\( \VoI_{\ACTIONS}\np{\va{\beliefbis}} = 0 \Rightarrow \) 
\( \PP \left\{ \proscal{ \va{\beliefbis} }{\action'-\action} \leq 0 \eqsepv
\right. \)
\( \left. \forall \action' \in \ACTIONS \eqsepv
\forall \action \in \overline{\FACE} \right\} = 1 \). 
Now, since \( \overline{\FACE}= \FACE_{\ACTIONS}(\priorbelief) \), 
we finally get that 
\( \VoI_{\ACTIONS}\np{\va{\beliefbis}} = 0 \Rightarrow \)
\( \PP \{ \proscal{ \va{\beliefbis} }{\action'-\action} \leq 0, \)
\( \forall \action' \in \ACTIONS, \; 
\forall \action \in \FACE_{\ACTIONS}(\priorbelief) \} =~1 \). 
In other words, we have obtained that, 
by definition~\eqref{eq:normal_cone_signed} 
of the {normal cone}~$\NORMAL_{\ACTIONS}(\action)$:
\( \VoI_{\ACTIONS}\np{\va{\beliefbis}} = 0 \Rightarrow \)
\( \va{\beliefbis} \in 
\bigcap_{ \action \in \FACE_{\ACTIONS}(\priorbelief) } \NORMAL_{\ACTIONS}(\action) 
\eqsepv \PP-\mathrm{a.s.} \). 
Since \( \va{\beliefbis} \in \BELIEFS \), we conclude
by~\eqref{eq:ConfidenceSet_bis} that 
\begin{equation*}
\VoI_{\ACTIONS}\np{\va{\beliefbis}} = 0 \Rightarrow
\va{\beliefbis} \in 
\bigcap_{ \action \in \FACE_{\ACTIONS}(\belief) } \NORMAL_{\ACTIONS}(\action) 
\cap \BELIEFS =
\bigcap_{ \action \in \OptimalActions(\belief) } \RevealedBeliefs(\action) 
= \ConfidenceSet(\belief) \eqfinp
  \end{equation*}
\end{subequations}
Revisiting the proof backward, or using~\eqref{eq:ConfidenceSet_ter_b},
we easily see that \( \va{\beliefbis} \in 
\ConfidenceSet(\belief) \eqsepv \PP-\mathrm{a.s.} 
\Rightarrow \) \( \VoI_{\ACTIONS}\np{\va{\beliefbis}} = 0 \). 
This ends the proof.
\end{proof}
\smallskip

\begin{proof}[Proof of Theorem~\ref{th:Valuable_information}]

Let $\va{\beliefbis}$ be an information structure
as in~\eqref{eq:information_structure}. 
\smallskip

First, we show the upper estimate 
\( \ConstantSupspeciale \EE d\bp{\va{\beliefbis},\ConfidenceSet(\priorbelief)}
\geq \VoI_{\ACTIONS}\np{\va{\beliefbis} } \)
in~\eqref{eq:Valuable_information}.
For this purpose, we consider $\action \in  \ACTIONS$ and we show that 
the function \( \varphi_{\action} \) in~\eqref{eq:shifted_value_function}
is such that
\begin{equation}
\varphi_{\action}\np{\beliefbis} \leq 
\sup_{\action' \in  \ACTIONS} \norm{\action-\action'}
\inf_{\belief \in \ConfidenceSet(\priorbelief)} \Vert \belief-\beliefbis\Vert \eqfinp 
\label{eq:shifted_value_function_upperbound}
\end{equation}
Indeed, we have that, for any \( \belief \in \ConfidenceSet(\priorbelief) \), 
\begin{subequations}
  \begin{align*}
 \varphi_{\action}(\beliefbis) &=
\varphi_{\action}(\beliefbis) - \varphi_{\action}(\belief)
\mtext{ by~\eqref{eq:shifted_value_function_=0}
since  $\belief \in \ConfidenceSet(\priorbelief)$ } \\
&= \sigma_\ACTIONS(\beliefbis)-\sigma_\ACTIONS(\belief) -
\proscal{\beliefbis-\belief}{\action} 
\mtext{ by~\eqref{eq:shifted_value_function}  } \\
&  =
\sigma_{\ACTIONS-\action}(\beliefbis)-\sigma_{\ACTIONS-\action}(\belief)
\mtext{ by~\eqref{eq:support_function} } \\
& \leq 
\sup_{\action' \in  \ACTIONS-\action} \norm{\action'} \times 
\Vert \belief-\beliefbis\Vert 
\mtext{ by~\eqref{eq:support_function} } 
= \sup_{\action' \in  \ACTIONS} \norm{\action-\action'} \times 
\Vert \belief-\beliefbis\Vert \eqfinp
  \end{align*}
\end{subequations}
By taking the infimum with respect to all
\( \belief \in \ConfidenceSet(\priorbelief) \), 
we obtain~\eqref{eq:shifted_value_function_upperbound}.
Then, we deduce that 
\begin{subequations}
  \begin{align*}
\VoI_{\ACTIONS}\np{\va{\beliefbis} } & = 
\EE \left[ \varphi_{\action}(\va{\beliefbis}) \right] \eqsepv
\forall \action \in  \ACTIONS \mtext{ by~\eqref{eq:VoI_ter} } \\
& = \inf_{\action \in  \ACTIONS}
\EE \left[ \varphi_{\action}(\va{\beliefbis}) \right] 
\leq \inf_{\action \in  \ACTIONS}\sup_{\action' \in  \ACTIONS} \norm{\action-\action'}
\times \EE \Bc{ 
\inf_{\belief \in \ConfidenceSet(\priorbelief)} \Vert \belief-\beliefbis\Vert }
\mtext{ by~\eqref{eq:shifted_value_function_upperbound}. }
  \end{align*}
\end{subequations}
With \( \ConstantSupspeciale = 
\inf_{\action \in  \ACTIONS}\sup_{\action' \in  \ACTIONS} \norm{\action-\action'} \)
and~\eqref{eq:distance}, this gives the upper estimate 
\( \ConstantSupspeciale \EE d\bp{\va{\beliefbis},\ConfidenceSet(\priorbelief)}
\geq \VoI_{\ACTIONS}\np{\va{\beliefbis} } \)
in~\eqref{eq:Valuable_information}.
\smallskip

Second, we show the lower estimate 
\( \VoI_{\ACTIONS}\np{\va{\beliefbis} } \geq 
\ConstantInfEpsilon \PP\{ \va{\beliefbis} \not\in 
{\BELIEFS_{\ACTIONS,\varepsilon}^{\textrm{c}}}(\priorbelief) \} \)
in~\eqref{eq:Valuable_information}.
We consider an open subset~${\cal Q} $ of~$\BELIEFS$ that contains
the {confidence set}~\( \ConfidenceSet(\belief) \),
that is, \( \ConfidenceSet(\priorbelief) \subset {\cal Q} \).
By Lemma~\ref{lem:strictly_positive} right below, 
there exists an \( \action \in \FACE_{\ACTIONS}(\priorbelief) \)
such that the continuous function~$ \varphi_{\action}$ 
in~\eqref{eq:shifted_value_function} is strictly positive on
\( \ConfidenceSet(\priorbelief)^c \).
As \( {\cal Q}^c \subset \ConfidenceSet(\priorbelief)^c \)
and \( {\cal Q}^c \) is a closed subset of the compact~$\BELIEFS$, 
we can define
\( 
\ConstantInf = \inf_{\belief \not\in {\cal Q} } \varphi_{\action}(\belief)> 0 
\). 
We deduce that 
\begin{subequations}
  \begin{align*}
\VoI_{\ACTIONS}\np{\va{\beliefbis} } & = 
\EE \left[ \varphi_{\action}(\va{\beliefbis}) \right] 
\mtext{ by~\eqref{eq:VoI_ter} } \\
& = \EE \left[ 
\1_{\va{\beliefbis} \in \ConfidenceSet(\priorbelief) } \varphi_{\action}(\va{\beliefbis})
+ \1_{\va{\beliefbis} \not\in \ConfidenceSet(\priorbelief) } 
\varphi_{\action}(\va{\beliefbis}) \right] \\
& = \EE \left[ 
 \1_{\va{\beliefbis} \not\in \ConfidenceSet(\priorbelief) } 
\varphi_{\action}(\va{\beliefbis}) \right] 
\mtext{ by~\eqref{eq:shifted_value_function_=0} } \\
& \geq \EE \left[ 
 \1_{\va{\beliefbis} \not\in {\cal Q} }
\varphi_{\action}(\va{\beliefbis}) \right] \geq \EE \left[ 
 \1_{\va{\beliefbis} \not\in {\cal Q} }
\ConstantInf  \right] = \ConstantInf 
\PP \{ \va{\beliefbis} \not\in {\cal Q} \} \eqfinp
  \end{align*}
\end{subequations}
With \( {\cal Q} = {\BELIEFS_{\ACTIONS,\varepsilon}^{\textrm{c}}}(\priorbelief) \),
we put 
\( 
\ConstantInfEpsilon =
\inf_{\belief \not\in {\BELIEFS_{\ACTIONS,\varepsilon}^{\textrm{c}}}(\priorbelief) } 
\varphi_{\action}(\belief) > 0 
\). 

This ends the proof. 
\end{proof}
\smallskip

\begin{lemma}
 There exists at least one 
\( \action \in \FACE_{\ACTIONS}(\priorbelief) \)
such that the function~$ \varphi_{\action}$ 
in~\eqref{eq:shifted_value_function}
is strictly positive on
the complementary set~\( \ConfidenceSet(\priorbelief)^c \).
\label{lem:strictly_positive}
\end{lemma}

\begin{proof}
We consider two cases, depending whether 
 \( \FACE_{\ACTIONS}(\priorbelief) \) is a singleton or not.

Suppose that \( \FACE_{\ACTIONS}(\priorbelief) \) is a singleton
$\{ \action \}$. By~\eqref{eq:ConfidenceSet_ter_b}, we have that 
\( 
\beliefbis \not\in \ConfidenceSet(\priorbelief) 
\iff \varphi_{\action}(\beliefbis) > 0 
\). 

Suppose that \( \FACE_{\ACTIONS}(\priorbelief) \) is a not singleton.
Recall that the \emph{affine hull}~$\aff(\SUBSET)$ 
of a subset~$\SUBSET$ of~$\RR^{\NATURE}$
is the intersection of all affine manifolds containing~$\SUBSET$,
and that the \emph{relative interior}~$\ri(\Convex)$ 
of a nonempty convex set~${\Convex} \subset \RR^{\NATURE}$ 
is the nonempty interior of~${\Convex}$ for the topology relative to the 
{affine hull}~$\aff(\Convex)$ \cite[p.~103]{Hiriart-Ururty-Lemarechal-I:1993}.
We prove that any \( \action \in \ri\bp{\FACE_{\ACTIONS}(\beliefbis)} \)
answers the question. 
Let \( \action \in \ri\bp{\FACE_{\ACTIONS}(\beliefbis)} \) be fixed.
For any \( \beliefbis \not\in \ConfidenceSet(\priorbelief) \),
by~\eqref{eq:ConfidenceSet_ter_a} we have that 
\( \FACE_{\ACTIONS}(\priorbelief) \not\subset\FACE_{\ACTIONS}(\beliefbis) \).
Therefore, there exists \( \bar\action \in \FACE_{\ACTIONS}(\priorbelief) \)
such that \( \bar\action \not\in \FACE_{\ACTIONS}(\beliefbis) \),
that is, such that \( \sigma_\ACTIONS(\beliefbis) > 
\proscal{\beliefbis}{\bar\action} \).
As \( \action \in \ri\bp{\FACE_{\ACTIONS}(\beliefbis)} \),
there exists \( \action' \in \ri\bp{\FACE_{\ACTIONS}(\beliefbis)} \)
such that \( \action = \lambda \action' + (1-\lambda) \bar\action \)
for a certain \( \lambda \in ]0,1[ \).
Since 
\( \sigma_\ACTIONS(\beliefbis) \geq \proscal{\beliefbis}{\action'} \) 
(by definition~\eqref{eq:support_function} of 
\( \sigma_\ACTIONS \)) and 
\( \sigma_\ACTIONS(\beliefbis) > \proscal{\beliefbis}{\bar\action} \)
(as \( \bar\action \not\in \FACE_{\ACTIONS}(\beliefbis) \)),
we deduce that 
\( 
  \sigma_\ACTIONS(\beliefbis) = 
\lambda \sigma_\ACTIONS(\beliefbis) + (1-\lambda) \sigma_\ACTIONS(\beliefbis) 
> \lambda \proscal{\beliefbis}{\action'} 
+ (1-\lambda) \proscal{\beliefbis}{\bar\action} = 
 \proscal{\beliefbis}{\action} \), 
where we used the property that \( \lambda \in ]0,1[ \).
Using the definition~\eqref{eq:shifted_value_function} 
of the function~$ \varphi_{\action}$, we have obtained that 
\( \beliefbis \not\in \ConfidenceSet(\priorbelief) 
\Rightarrow \varphi_{\action}(\beliefbis) > 0 \).

This ends the proof. 
\end{proof}

\subsubsection{Undecided}

\begin{proof}[Proof of Theorem~\ref{th:Indifferences}]

We prove the three inequalities in~\eqref{eq:Indifferences}.
\smallskip

I). We prove the upper inequality
\( \ConstantSupspeciale \EE \Vert \va{\beliefbis} - \priorbelief \Vert \geq 
\VoI_{\ACTIONS}\np{\va{\beliefbis}} \) in~\eqref{eq:Indifferences}.

By definition~\eqref{eq:support_function} of a support function, we have that 
\( \sigma_\ACTIONS(\cdot) \leq \norm{\ACTIONS} \times \norm{\cdot} \),
where \( \norm{\ACTIONS}= \sup
\{ \norm{\action} \eqsepv \action \in \ACTIONS \} < +\infty \).
Thus \( \ConstantSupspeciale = \norm{\ACTIONS} \) in the left hand side
inequality in~\eqref{eq:Indifferences}.
\smallskip

II). We prove the middle inequality
\( \VoI_{\ACTIONS}\np{\va{\beliefbis}} \geq 
\VoI_{\OptimalActions(\priorbelief)}\np{\va{\beliefbis}} \)
 in~\eqref{eq:Indifferences}.

For all \( \signed \in \SIGNED \), we have that 
\begin{subequations}
\begin{align}
  \sigma_\ACTIONS(\signed)-\sigma_\ACTIONS(\priorbelief) \geq & 
\sigma_{ \FACE_{\ACTIONS}(\priorbelief) } \np{\signed - \priorbelief} 
\label{eq:LowerBoundSupportFunction_a} 
\mtext{ by~\eqref{eq:property_subdifferential_of_the_support_function} 
since \( \FACE_{\ACTIONS}(\priorbelief) \not = \emptyset \) } \\
=& \proscal{\signed-\priorbelief}{\action} \eqsepv
\forall \action \in \FACE_{\ACTIONS}(\priorbelief) 
\mtext{ by definition of~$\sigma_{\FACE_{\ACTIONS}(\priorbelief) }$ } \\
=& \sigma_{\FACE_{\ACTIONS}(\priorbelief) } \np{\signed} - 
\sigma_{\FACE_{\ACTIONS}(\priorbelief) } \np{\priorbelief} 
\label{eq:LowerBoundSupportFunction_c} 
\mtext{ by definition of~$\sigma_{\FACE_{\ACTIONS}(\priorbelief) }$. }
\end{align}  
\label{eq:LowerBoundSupportFunction}
\end{subequations}
By taking the expectation~$\EE$, we obtain that 
\begin{subequations}
\begin{align}
\VoI_{\ACTIONS}\np{\va{\beliefbis}} =&
\EE  \left[ \sigma_{\ACTIONS}\np{\va{\beliefbis}}
-\sigma_{\ACTIONS}(\priorbelief) \right] 
\mtext{ by~\eqref{eq:VoI} and~\eqref{eq:value_function_support_function} }\\
\geq & \EE  \left[ \sigma_{\FACE_{\ACTIONS}(\priorbelief) } \np{\va{\beliefbis}
-\priorbelief} \right] 
\mtext{ by~\eqref{eq:LowerBoundSupportFunction_a} } 
\label{eq:LowerBoundVoIFace_b} \\ 
= & \EE  \left[ \sigma_{\FACE_{\ACTIONS}(\priorbelief) } \np{\va{\beliefbis}} - 
\sigma_{\FACE_{\ACTIONS}(\priorbelief)} \np{\priorbelief} \right] 
\mtext{ by~\eqref{eq:LowerBoundSupportFunction_c} } 
\label{eq:LowerBoundVoIFace_c} \\ 
=& \VoI_{\FACE_{\ACTIONS}(\priorbelief)}\np{\va{\beliefbis}} 
\mtext{ by~\eqref{eq:VoI} and~\eqref{eq:value_function_support_function}. }
\nonumber
\end{align}
\end{subequations}
This ends the proof of the middle inequality.
\smallskip

III). We prove the right hand side inequality 
\( \VoI_{\OptimalActions(\priorbelief)}\np{\va{\beliefbis}} \geq 
\EE \Vert \va{\beliefbis} - \priorbelief 
\Vert_{\IndifferenceKernel(\priorbelief)} \) in~\eqref{eq:Indifferences}. 

Let $n$ be the dimension of 
the affine hull~$\aff\bp{\FACE_{\ACTIONS}(\priorbelief)}$
of~$\FACE_{\ACTIONS}(\priorbelief)$, and let \( \action_1, \ldots, \action_n \) 
be $n$~actions in~$\FACE_{\ACTIONS}(\priorbelief)$
that generate ~$\aff\bp{\FACE_{\ACTIONS}(\priorbelief)}$. We put
\begin{equation}
T=\{\action_1, \ldots, \action_n \} \subset \FACE_{\ACTIONS}(\priorbelief) 
\mtext{ so that\,}\aff\bp{\FACE_{\ACTIONS}(\priorbelief)} =
\aff \{\action_1, \ldots, \action_n \} = \aff\np{T} \eqfinp 
\label{eq:affT=affFace}
\end{equation}
We will now show that 
\( 
\Vert \cdot \Vert_{\IndifferenceKernel(\priorbelief)}
=  \frac{1}{n} \sigma_{ T - T }\np{\cdot}
\) 
is a seminorm with kernel 
\( \np{\FACE_{\ACTIONS}(\priorbelief)-\FACE_{\ACTIONS}(\priorbelief)}^\perp \)
that satisfies the right hand side
inequality in~\eqref{eq:Indifferences}. 
\smallskip

First, the support function $\sigma_{T-T}$ is a seminorm 
with kernel \( \np{T-T}^\perp \), as easily seen.
Now, we also easily see that, for any subset~$S \subset \RR^{\NATURE}$,
one has 
\( 
\np{S-S}^\perp = \bp{ \aff\np{S-S} }^\perp = 
\bp{ \aff\np{S}-\aff\np{S} }^\perp 
\). 
Using these equalities with $S=T$ and $S=\FACE_{\ACTIONS}(\priorbelief)$,
we deduce that \( \np{T-T}^\perp = 
\np{\FACE_{\ACTIONS}(\priorbelief)-\FACE_{\ACTIONS}(\priorbelief)}^\perp \),
since \( \aff\np{T}=\aff\bp{\FACE_{\ACTIONS}(\priorbelief)} \)
by~\eqref{eq:affT=affFace}.
Second, we show that the right hand side
inequality in~\eqref{eq:Indifferences} is satisfied.
We have
\begin{subequations}
\begin{align*}
\VoI_{\ACTIONS}\np{\va{\beliefbis}} & \geq 
\EE \left[ \sigma_{\FACE_{\ACTIONS}(\priorbelief)} (\va{\beliefbis}-\priorbelief)
\right] \mtext{ by~\eqref{eq:LowerBoundVoIFace_b} } \\ 
&\geq \EE \left[ \sigma_{T} (\va{\beliefbis}-\priorbelief) \right]
\tag{because $T \subset \FACE_{\ACTIONS}(\priorbelief) $ 
and support functions~\eqref{eq:support_function} 
are monotone with respect to set inclusion} 
\\
&=\EE \left[ \sigma_{T} (\va{\beliefbis}-\priorbelief) 
- \proscal{ \va{\beliefbis} - \priorbelief }{\action} \right] 
\eqsepv \forall \action \in \ACTIONS 
\mtext{ because 
\( \EE \left[ \proscal{ \va{\beliefbis} - \priorbelief }{\action} \right] 
= 0 \). } \\
&=\EE \left[ \sigma_{T-\action} (\va{\beliefbis}-\priorbelief)  \right] 
\eqsepv \forall \action \in \ACTIONS 
\mtext{ because $\sigma_{T-\action}=\sigma_{T + \{-\action\}}
=\sigma_{T} + \sigma_{\{-\action\}}$.}
\end{align*}
\end{subequations}
Indeed, support functions transform a Minkowski sum of sets
into a sum of support functions
\cite[p.~226]{Hiriart-Ururty-Lemarechal-I:1993}. Using again this property, we
obtain that 
\(   \VoI_{\ACTIONS}\np{\va{\beliefbis}}  \geq \frac{1}{n} 
\sum_{i=1}^n  \EE \left[ \sigma_{T-\action_i}
  (\va{\beliefbis}-\priorbelief) \right] \)
\( = \frac{1}{n}  \EE \left[ 
\sigma_{ \sum_{i=1}^n \np{T-\action_i} } 
(\va{\beliefbis}-\priorbelief) \right] \). 
Now, as \( T=\{\action_1, \ldots, \action_n \} \), 
it is easy to see that the sum 
\( \sum_{i=1}^n \np{T-\action_i} \) contains any
element of the form \( \action_k - \action_l =\)
\( (\action_1 - \action_1) + \cdots + 
   (\action_{l-1} - \action_{l-1}) + (\action_k- \action_l) + 
   (\action_{l+1} - \action_{l+1}) + \cdots + (\action_n - \action_n)
\in \sum_{i=1}^n \np{T-\action_i} \). 
As support functions are monotone with respect to set inclusion, 
we deduce that 
\( 
\sigma_{ \sum_{i=1}^n \np{T-\action_i} } \geq
\sigma_{ \{ \action_k - \action_l , k,l= 1,\ldots, n \} } 
= \sigma_{T-T} 
\) 
and that
\( 
  \VoI_{\ACTIONS}\np{\va{\beliefbis}}  \geq \frac{1}{n} \EE \left[ 
\sigma_{ \{ \action_k - \action_l , k,l= 1,\ldots, n \} } 
(\va{\beliefbis}-\priorbelief) \right] = 
\frac{1}{n} \EE \left[ \sigma_{ T - T }
(\va{\beliefbis}-\priorbelief) \right]  
= \EE \Vert \va{\beliefbis} - \priorbelief 
\Vert_{\IndifferenceKernel(\priorbelief)} 
\). 

This ends the proof. 
\end{proof}

\subsubsection{Flexible}

\begin{proof}[Proof of Proposition~\ref{pr:Flexible_decisions}]

All the reminders on geometric convex analysis 
in~Sect. \ref{Recalls_on_geometric_convex_analysis} 
were done with outer normal vectors
belonging to the unit sphere of signed measures. 
Now, as we work with beliefs --- positive measures of mass~1 ---
we are going to adapt these concepts. 
We consider the diffeomorphism 
\begin{equation}
  \nu : \sphere \cap \RR_+^{\NATURE} \to \BELIEFS \eqsepv \signed \mapsto  
\frac{\signed}{\proscal{\signed}{1}} \eqfinv
\label{eq:nu}
\end{equation}
that maps unit positive measures into probability measures, 
with inverse
\( 
  \nu^{-1} : \BELIEFS \to \sphere \cap \RR_+^{\NATURE} \eqsepv 
\belief \mapsto \frac{\belief}{\norm{\belief}} 
\). 

Since, by assumption, 
the {action set}~$\ACTIONS$ has boundary~$\partial\ACTIONS$ 
which is a~$C^2$ submanifold of~$\RR^{\NATURE}$, 
we know by Proposition~\ref{pr:normal_circ_face} that 
the spherical image map~\( \normal_{\ACTIONS} : \partial\ACTIONS \to \sphere \) 
in~\eqref{eq:spherical_image_map} is well defined, is of class~$C^1$,
and has for right inverse the reverse spherical image 
map~$\face_{\ACTIONS} : \mathrm{regn}\np{\ACTIONS} \to \partial\ACTIONS$
in~\eqref{eq:reverse_spherical_image_map}, 
that is, \( \normal_{\ACTIONS} \circ \face_{\ACTIONS} =
\mathrm{Id}_{\mathrm{regn}\np{\ACTIONS}} \). 

The \emph{set of relevant regular points} is the subset of 
the set~\( \mathrm{reg}\np{\ACTIONS} \) of regular points defined by
\begin{equation}
\action \in \mathrm{reg}^{+}\np{\ACTIONS} \iff \exists \belief \in \BELIEFS 
\eqsepv \NORMAL_{\ACTIONS}(\action)=\RR_+ \belief \eqfinp 
\label{eq:set_of_relevant_regular_points}
\end{equation}
For a regular action~$\action \in \mathrm{reg}^{+}\np{\ACTIONS}$, there is only one 
probability~$\belief \in \BELIEFS~$ such that \( \NORMAL_{\ACTIONS}(\action)=\RR_+ \belief
\), and it is \( \belief = \nu\bp{\normal_{\ACTIONS}\np{\action}} \). 
We have 
\( 
  \action \in \mathrm{reg}^{+}\np{\ACTIONS} \Rightarrow 
\NORMAL_{\ACTIONS}(\action) = \RR_+ \nu\bp{\normal_{\ACTIONS}\np{\action}}
\) where 
\( \nu\bp{\normal_{\ACTIONS}\np{\action}} \in \BELIEFS
\). 
The \emph{set of regular probabilities} is
\( 
\mathrm{regn}^{+}\np{\ACTIONS} = \Bp{ \RR_+^* \mathrm{regn}\np{\ACTIONS} }
\cap \BELIEFS 
\). 
For a regular probability~$\belief \in \mathrm{regn}^{+}\np{\ACTIONS}$, 
there is only one action~$\action \in \partial\ACTIONS$ such that 
\( \FACE_{\ACTIONS}(\belief) = \{ \action \} \), and it is 
\( \action = \face_{\ACTIONS}\bp{\nu^{-1}\np{\belief}} \). 
Indeed, by definition~\eqref{eq:face_signed} of the (exposed) face, we have that 
\( 
  \FACE_{\ACTIONS}(\lambda \signed) = \FACE_{\ACTIONS}(\signed) 
\eqsepv \forall \lambda \in \RR_+^* \eqsepv 
\forall \signed \in \SIGNED \eqsepv \signed \not = 0  
\). 
Therefore, we have that 
\begin{equation}
 \belief \in \mathrm{regn}^{+}\np{\ACTIONS} \Rightarrow 
\FACE_{\ACTIONS}(\belief) = \{ \face_{\ACTIONS}\bp{\nu^{-1}\np{\belief}} \} \eqfinp
\label{eq:reverse_probability_image_map_characterization}
\end{equation}
The following mappings are well defined:
\( 
  \nu \circ \normal_{\ACTIONS} : \mathrm{reg}^{+}\np{\ACTIONS} \to \BELIEFS
\) and 
\( \face_{\ACTIONS} \circ \nu^{-1}  : \mathrm{regn}^{+}\np{\ACTIONS} \to 
\partial\ACTIONS \),
and we have that 
\( 
\np{\nu \circ \normal_{\ACTIONS}} \circ \np{\face_{\ACTIONS} \circ \nu^{-1}} =
\mathrm{Id}_{\mathrm{regn}^{+}\np{\ACTIONS}} 
\). 

\begin{itemize}
\item 
Item~\ref{it:curved_face}~$\Rightarrow$ Item~\ref{it:curved_diffeomorphism}.\\
Suppose that the face~$\FACE_{\ACTIONS}(\priorbelief)$ is a singleton
\( \{ \action\opt \} \)
and the curvature of the boundary~$\partial\ACTIONS$ of payoffs 
at~$\action\opt$ is positive.
Since, by assumption, 
the {action set}~$\ACTIONS$ has boundary~$\partial\ACTIONS$ 
which is a~$C^2$ submanifold of~$\RR^{\NATURE}$, we know that
the spherical image map~\( \normal_{\ACTIONS} \) 
in~\eqref{eq:spherical_image_map} is defined 
over the whole boundary~$\partial\ACTIONS$ and is of class~$C^1$,
and its differential is the {Weingarten map}.
As the curvature of the boundary~$\partial\ACTIONS$ of payoffs 
at~$\action\opt$ is positive, 
the {Weingarten map}~$T_{\action\opt}\normal_{\ACTIONS}$
is of maximal rank at~$\action\opt$ \cite[p.~115]{Schneider:2014}.
Therefore, by the  inverse function theorem, 
there exists an open neighborhood~\(\neighborhood{\Action}\) 
of~$\action\opt$ in~$\ACTIONS$ such that 
\( \normal_{\ACTIONS}\np{\neighborhood{\Action}} \) is an open neighborhood
of \( \normal_{\ACTIONS}\bp{ \action\opt } \) in~$\sphere$,
and such that the restriction \( \normal_{\ACTIONS} : 
\neighborhood{\Action} \to \normal_{\ACTIONS}\np{\neighborhood{\Action}} \)
of the spherical image map in~\eqref{eq:spherical_image_map} is a
diffeomorphism.
By item iii) in Proposition~\ref{pr:normal_circ_face},  
we have that 
\( \normal_{\ACTIONS}\bp{ \action\opt } = 
\frac{\priorbelief}{\norm{\priorbelief}} \) and
the local inverse coincides with the 
restriction \( \face_{\ACTIONS} : \normal_{\ACTIONS}\np{\neighborhood{\Action}} 
\to \neighborhood{\Action} \) of the {reverse spherical image map}
in~\eqref{eq:reverse_spherical_image_map}. 
As \( \normal_{\ACTIONS}\np{\neighborhood{\Action}} \) is an open neighborhood
of~\(\frac{\priorbelief}{\norm{\priorbelief}} \) in~$\sphere$,
and as the prior belief~$\priorbelief$ has full support, we deduce that
\( \nu\bp{\normal_{\ACTIONS}\np{\neighborhood{\Action}}} \) is an open neighborhood
of~\(\priorbelief\) in~$\BELIEFS$,
where the diffeomorphism~$\nu$ is defined in~\eqref{eq:nu}. 
We easily deduce that \( \face_{\ACTIONS} \circ \nu^{-1} : 
\nu\bp{\normal_{\ACTIONS}\np{\neighborhood{\Action}}} \to
\neighborhood{\Action}\) is a diffeomorphism. 
By~\eqref{eq:reverse_probability_image_map_characterization}, 
we conclude that \( \face_{\ACTIONS} \circ \nu^{-1} \) is the restriction of the 
set-valued mapping \( \FACE_{\ACTIONS} :  \BELIEFS \rightrightarrows
\ACTIONS \), \( \belief \mapsto \FACE_{\ACTIONS}(\belief) \) 
in~\eqref{eq:face_set-valued_mapping}.

\item 
Item~\ref{it:curved_diffeomorphism}~$\Rightarrow$ Item~\ref{it:curved_value_function}.\\
Suppose that the set-valued mapping \( \FACE_{\ACTIONS} :  \BELIEFS \rightrightarrows
\ACTIONS \), \( \belief \mapsto \FACE_{\ACTIONS}(\belief) \) 
in~\eqref{eq:face_set-valued_mapping} is a local diffeomorphism
at~$\priorbelief$.
By definition~\eqref{eq:set_of_regular_signed_measures} of the 
{set of regular unit signed measures}, 
there exists an open neighborhood~$\neighborhood{\beliefbis}$ 
of~$\priorbelief$ in~$\BELIEFS$ such that 
\( \neighborhood{\beliefbis} \subset \mathrm{regn}^{+}\np{\ACTIONS} \),
where the {set of relevant regular points} is defined 
in~\eqref{eq:set_of_relevant_regular_points}.
In addition, the mapping \( \face_{\ACTIONS} \circ \nu^{-1} : 
\neighborhood{\beliefbis} \to
\face_{\ACTIONS}\bp{\nu^{-1}\np{\neighborhood{\beliefbis}}} \) 
is a diffeomorphism. 

As \( \FACE_{\ACTIONS}(\belief) = 
\{ \face_{\ACTIONS}\bp{\nu^{-1}\np{\belief}} \} \),
for all beliefs~\( \belief \in \neighborhood{\beliefbis} \), 
we know that the support function~$\sigma_\ACTIONS$ is differentiable and that 
its gradient is \( \nabla_{\belief}\sigma_\ACTIONS = 
\face_{\ACTIONS}\bp{\nu^{-1}\np{\belief}} \)
\cite[p.~251]{Hiriart-Ururty-Lemarechal-I:1993}. 
As  \( \face_{\ACTIONS} \circ \nu^{-1} \) is a local diffeomorphism
at~$\priorbelief$, and as the mapping~$\nu$ in~\eqref{eq:nu} is a
diffeomorphism, we deduce that the support function~$\sigma_\ACTIONS$ is 
twice differentiable with Hessian having full rank. 
As the value function~$\Value_{\ACTIONS}$ is the restriction
of~$\sigma_\ACTIONS$ to~$\BELIEFS$, we conclude that~$\Value_{\ACTIONS}$ is 
twice differentiable at~$\priorbelief$ and the Hessian is positive definite.

\item 
Item~\ref{it:curved_value_function}~$\Rightarrow$ Item~\ref{it:curved_face}.\\
Suppose that the value function~$\Value_{\ACTIONS}$ 
is twice differentiable at~$\priorbelief$
and the Hessian is positive definite.

On the one hand, as the prior~$\priorbelief$ has full support,
there exists an open neighborhood~$\neighborhood{\beliefbis}$ 
of~$\priorbelief$ in~$\BELIEFS$ such that~$\Value_{\ACTIONS}$ 
is differentiable on~$\neighborhood{\beliefbis}$.
On the other hand, as the support function~$\sigma_\ACTIONS$ 
is positively homogeneous,
and by~\eqref{eq:value_function_support_function}, we have that
\begin{equation}
\sigma_\ACTIONS\np{\signed} = \proscal{\signed}{1} \times 
  \bp{\Value_{\ACTIONS} \circ \nu} \np{\signed} \eqsepv 
\forall \signed \in \sphere \cap \RR_+^{\NATURE} \eqfinp
\label{eq:curved_value_function_support}
\end{equation}
Therefore, as the mapping~$\nu$ in~\eqref{eq:nu} is a diffeomorphism, 
the support function~$\sigma_\ACTIONS$ is differentiable on the 
open neighborhood~$\nu^{-1}\np{\neighborhood{\beliefbis}}$ 
of~$\nu^{-1}\np{\priorbelief}=
\frac{\priorbelief}{\norm{\priorbelief}}$ in~$\sphere \cap \RR_+^{\NATURE}$.

Since, on the one hand, 
a convex function with effective domain~$\RR^{\NATURE}$ is differentiable
at~$\signed$ if and only if the subdifferential at~$\signed$
is a singleton \cite[p.~251]{Hiriart-Ururty-Lemarechal-I:1993},
and, on the other hand, the face~$\FACE_{\ACTIONS}(\signed)$ is
the subdifferential at~$\signed$ of the support function~$\sigma_\ACTIONS$ 
\cite[p.~258]{Hiriart-Ururty-Lemarechal-I:1993},
we conclude that the face~$\FACE_{\ACTIONS}(\signed)$ of~$\ACTIONS$ 
in the direction~$\signed \in \nu^{-1}\np{\neighborhood{\beliefbis}}$ 
is a singleton.

Therefore, by definition~\eqref{eq:set_of_regular_signed_measures} of the 
{set of regular unit signed measures}, we have that 
\( \nu^{-1}\np{\neighborhood{\beliefbis}} \subset \mathrm{regn}\np{\ACTIONS} \).
In addition, the restriction \( \face_{\ACTIONS} : 
\nu^{-1}\np{\neighborhood{\beliefbis}} \to
\face_{\ACTIONS}\bp{\nu^{-1}\np{\neighborhood{\beliefbis}}} \) 
of the {reverse spherical image map}
in~\eqref{eq:reverse_spherical_image_map} is well defined, and we have that 
\( 
\nabla_{\signed}\sigma_\ACTIONS = \face_{\ACTIONS}\np{\signed} \eqsepv 
\forall \signed \in \nu^{-1}\np{\neighborhood{\beliefbis}} 
\). 
Therefore, the mapping \( \face_{\ACTIONS} : 
\nu^{-1}\np{\neighborhood{\beliefbis}} \to
\face_{\ACTIONS}\bp{\nu^{-1}\np{\neighborhood{\beliefbis}}} \) 
is differentiable at~$\nu^{-1}\np{\priorbelief}=
\frac{\priorbelief}{\norm{\priorbelief}}$, and has full rank.
Indeed, $\sigma_\ACTIONS$ is twice differentiable at~$\nu^{-1}\np{\priorbelief}=
\frac{\priorbelief}{\norm{\priorbelief}}$,
and the Hessian is positive definite.
This comes from~\eqref{eq:curved_value_function_support}, where 
the mapping~$\nu$ in~\eqref{eq:nu} is a~$C^{\infty}$ diffeomorphism
and the value function~$\Value_{\ACTIONS}$ 
is twice differentiable at~$\priorbelief$ with positive definite Hessian.

As \( \face_{\ACTIONS} \) is is differentiable 
at~$\frac{\priorbelief}{\norm{\priorbelief}}$ and has full rank, 
the {reverse Weingarten map}~$T_{\signed}\face_{\ACTIONS}$ 
in~\eqref{eq:reverse_Weingarten_map} is well defined and has full rank.
Therefore, the principal radii of curvature of~$\ACTIONS$ 
at~$\frac{\priorbelief}{\norm{\priorbelief}}$ are positive.
Letting \( \action\opt = 
\face_{\ACTIONS}\bp{\frac{\priorbelief}{\norm{\priorbelief}}} \), 
we conclude that \( \FACE_{\ACTIONS}\np{\priorbelief} = \{ \action\opt \} \)
and that the curvature of the boundary~$\partial\ACTIONS$ of payoffs 
at~$\action\opt$ is positive.
\end{itemize}
This ends the proof.
\end{proof}
\smallskip

\begin{proof}[Proof of Theorem~\ref{th:Flexible_decisions}]

We suppose that the value function~$\Value_{\ACTIONS}$  in~\eqref{eq:value_function}
is twice differentiable at~$\priorbelief$, with positive definite Hessian.
We denote \( \FACE_{\ACTIONS}\np{\priorbelief} = \{ \action\opt \} \). 

First, we show that the function 
\( 
  g(\belief) = \frac{ \Value_{\ACTIONS}(\belief) -
    \Value_{\ACTIONS}(\priorbelief) - 
\proscal{\belief-\priorbelief}{\action\opt} }%
{ \norm{\belief-\priorbelief}^2 } 
\) 
is continuous and positive on~$\BELIEFS$. 
Indeed, $g$ is continuous on~$\BELIEFS \backslash \{\priorbelief \}$,
and also at~$\priorbelief$ since the value function~$\Value_{\ACTIONS}$ 
is twice differentiable at~$\priorbelief$. 
In addition, \( g(\priorbelief) > 0 \) since the Hessian 
of~$\Value_{\ACTIONS}$  at~$\priorbelief$ is positive definite. 
We have~$g \geq 0~$ on~$\BELIEFS \backslash \{\priorbelief \}$,
because \( \FACE_{\ACTIONS}\np{\priorbelief} = \{ \action\opt \} \) is the 
subdifferential at~$\priorbelief$ of the support function~$\sigma_\ACTIONS$,
and by~\eqref{eq:value_function_support_function}. 
We now prove by contradiction that~$g > 0$.
If there existed a belief \( \belief \neq \priorbelief \) such that 
\( g(\belief) = 0 \), we would have \( \Value_{\ACTIONS}(\belief) -
    \Value_{\ACTIONS}(\priorbelief) - 
\proscal{\belief-\priorbelief}{\action\opt} = 0 \);
this equality would then hold true over the whole segment 
\( [\belief , \priorbelief ] \), and we would conclude that the second
derivative of~$\Value_{\ACTIONS}$ at~$\priorbelief$ along the (nonzero)
direction \( \belief-\priorbelief \) would be zero; 
this would contradict the assumption that the Hessian 
of~$\Value_{\ACTIONS}$  at~$\priorbelief$ is positive definite. 
Therefore, we conclude that~$g > 0$.
Second, letting~$\ConstantSup >0$ and $\ConstantInf >0$ 
be the maximum and the minimum
of the function~$g>0$ on the compact set~$\BELIEFS$, 
we easily deduce~\eqref{eq:Flexible_decisions} from~\eqref{eq:VoI}.

This ends the proof.
\end{proof}

\end{document}